\documentclass[11pt]{amsart}
\pdfoutput=1
\usepackage{amsmath, amssymb, amsthm, amsbsy, mathtools, bbm, amsfonts, color, verbatim, array, floatflt, yfonts,mathrsfs}

\usepackage[margin=1.2in]{geometry}
\usepackage[all]{xy}
\usepackage[english]{babel}
\usepackage[shortlabels]{enumitem}
\usepackage{stmaryrd} 

\usepackage{graphicx, color, overpic}
\usepackage[
dvipsnames,svgnames,table]{xcolor}
\usepackage[pdfpagelabels]{hyperref}
\usepackage{multirow}
\usepackage{easybmat}
\usepackage{adjustbox}

\usepackage{thmtools,thm-restate}

\usepackage{csquotes}
 
\usepackage{caption}
\usepackage{subcaption}
\usepackage[capitalise,noabbrev,sort]{cleveref}

\usepackage{parskip}

\usepackage{microtype}

\usepackage{tikz}

\theoremstyle{plain}
\newtheorem{thm}{Theorem}[section]
\newtheorem*{thm*}{Theorem}
\newtheorem{prop}[thm]{Proposition}
\newtheorem*{prop*}{Proposition}
\newtheorem{lemma}[thm]{Lemma}
\newtheorem*{lemma*}{Lemma}
\newtheorem{corollary}[thm]{Corollary}

\theoremstyle{definition}
\newtheorem{definition}[thm]{Definition}
\newtheorem{example}[thm]{Example}

\newtheorem{rmk}[thm]{Remark}

\crefalias{theoremx}{theorem}

\newcommand{\cC}{\mathcal{C}}
\newcommand{\cD}{\mathcal{D}}

\newcommand{\cG}{\mathcal{G}}

\newcommand{\R}{\mathbb{R}}
\newcommand{\Z}{\mathbb{Z}}
\newcommand{\N}{\mathbb{N}}
\newcommand{\C}{\mathbb{C}}

\newcommand{\cB}{\mathcal B}

\newcommand{\cT}{\mathcal T}

\newcommand{\WJ}{W_J}
\newcommand{\Wj}{W_{[j]}}
\newcommand{\JW}{\prescript{J}{}{W}} 
\newcommand{\jW}{\prescript{[j]}{}{W}} 

\newcommand{\G}{\Gamma} 

\newcommand{\NF}{\operatorname{NF}} 

\newcommand\KLbasis{\underline{\mathbf{H}}}
\newcommand\Hgen{\mathbf{H}}

\newcommand\Ngen{\mathbf{N}}

\newcommand{\Pxy}{P_{x,y}} 
 
\numberwithin{equation}{subsection}

\definecolor{amethyst}{rgb}{0.6, 0.4, 0.8}
\definecolor{kellygreen}{rgb}{0.3, 0.73, 0.09}
\definecolor{americanrose}{rgb}{1.0, 0.01, 0.24}

\begin{document}

\title[Cubulation of Bruhat graphs]{Cubulation of Bruhat graphs}

\author{Alex Bishop}
\address{School of Mathematical and Physical Sciences, University of Technology Sydney, Building 4 (CB04), 745 Harris Street, Broadway NSW 2007, Australia}
\email{alexbishop1234@gmail.com}

\author{Elizabeth Mili\'{c}evi\'{c}}
\address{Elizabeth Mili\'{c}evi\'{c}, Department of Mathematics \& Statistics, Haverford College, 370 Lancaster Avenue, Haverford, PA, USA}
\email{emilicevic@haverford.edu}

\author{Anne Thomas}
\address{Anne Thomas, School of Mathematics \& Statistics, Carslaw Building F07,  University of Sydney NSW 2006, Australia}
\email{anne.thomas@sydney.edu.au}

\thanks{EM was supported by NSF Grant DMS 2202017. AB and this research was supported in part by ARC Grant DP180102437, and AT was supported in part by ARC Grant FT250100160.}

\begin{abstract}
  For $(W,S)$ an arbitrary Coxeter system and any $y \in W$, we investigate the condition that the Bruhat graph for the interval $[1,y]$ can be cubulated, meaning roughly that this graph can be spanned by a product of subintervals of $\Z$. Results of Carrell--Peterson and Elias--Williamson imply that if $[1,y]$ can be cubulated, then the Kazhdan--Lusztig polynomial $\Pxy = 1$ for all $x \leq y$. We consider the converse to this result. For $(W,S)$ finite and $w_0$ the longest element in $W$, so that $P_{x,w_0} = 1$ for all $x \in W$, we use normal form forests to construct cubulations of $[1,w_0]$ in types $A$ and $B/C$. However, in some exceptional types, we determine elements $y \in W$ such that $P_{1,y} = 1$ but $[1,y]$ cannot be cubulated.  We then prove that if there are infinitely many $y \in W$ such that $[1,y]$ can be cubulated, then $(W,S)$ must be of type $\tilde{A}_n$ for some $n \geq 1$. Finally, for $(W,S)$ of type $\tilde{A}_2$, we exhibit a cubulation of $[1,y]$ for each of the infinitely many $y \in W$ such that $\Pxy = 1$ for all $x \leq y$.
\end{abstract}

\maketitle


\section{Introduction}

Let $(W,S)$ be an arbitrary Coxeter system and let $y \in W$. In this paper, we investigate the condition that the Bruhat graph for the interval $[1,y]$ can be \emph{cubulated}, meaning that it has a spanning subgraph which is a \emph{cubical lattice}; see Section \ref{sec:cubical} for the precise definitions, and \cref{fig:example-sublattice} for an example. 

\begin{figure}[h!]
	\begin{minipage}[t]{.35\linewidth}
		\includegraphics[width=\linewidth]{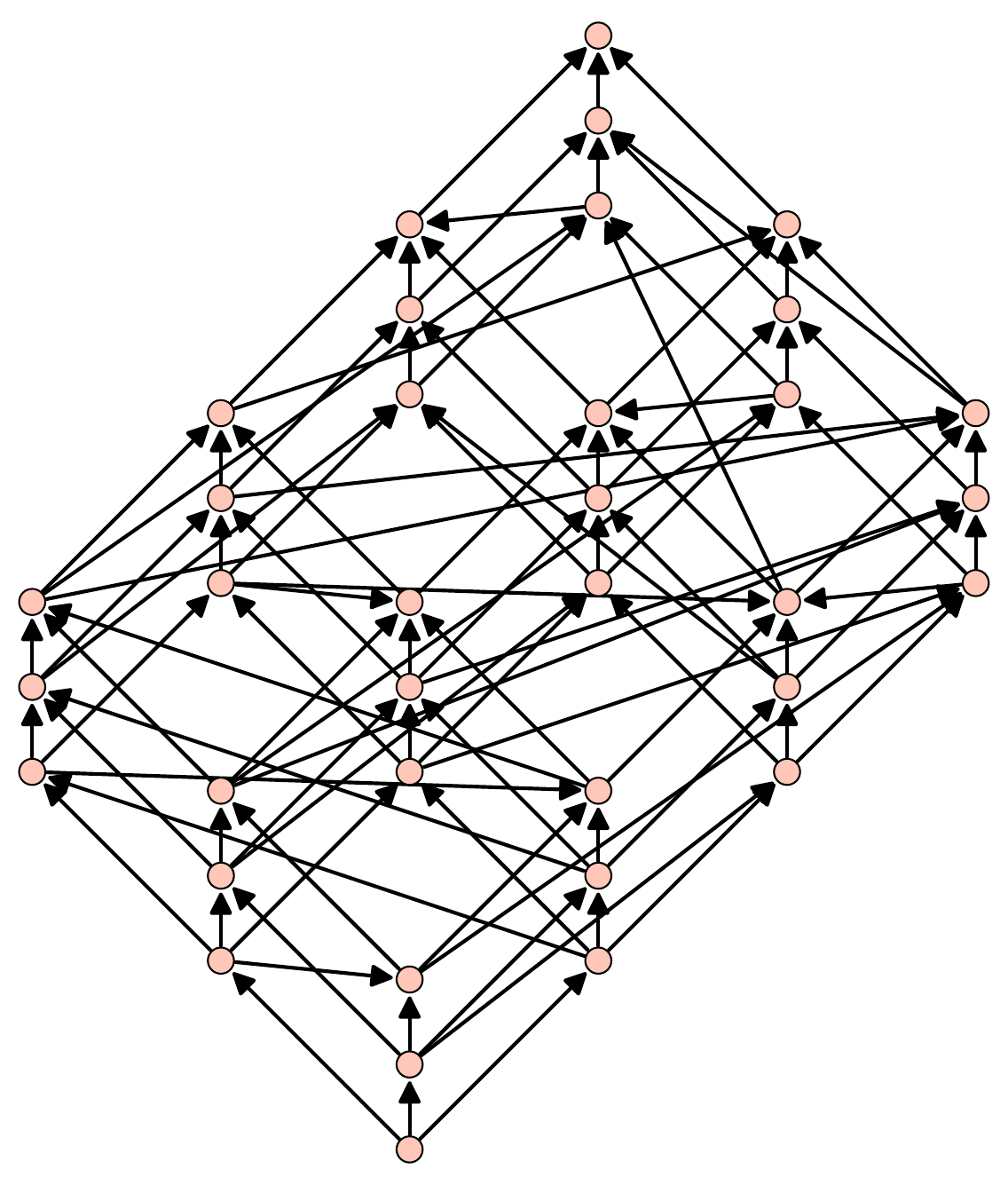}
	\end{minipage}
	~
	\begin{minipage}[t]{.35\linewidth}
		\includegraphics[width=\linewidth]{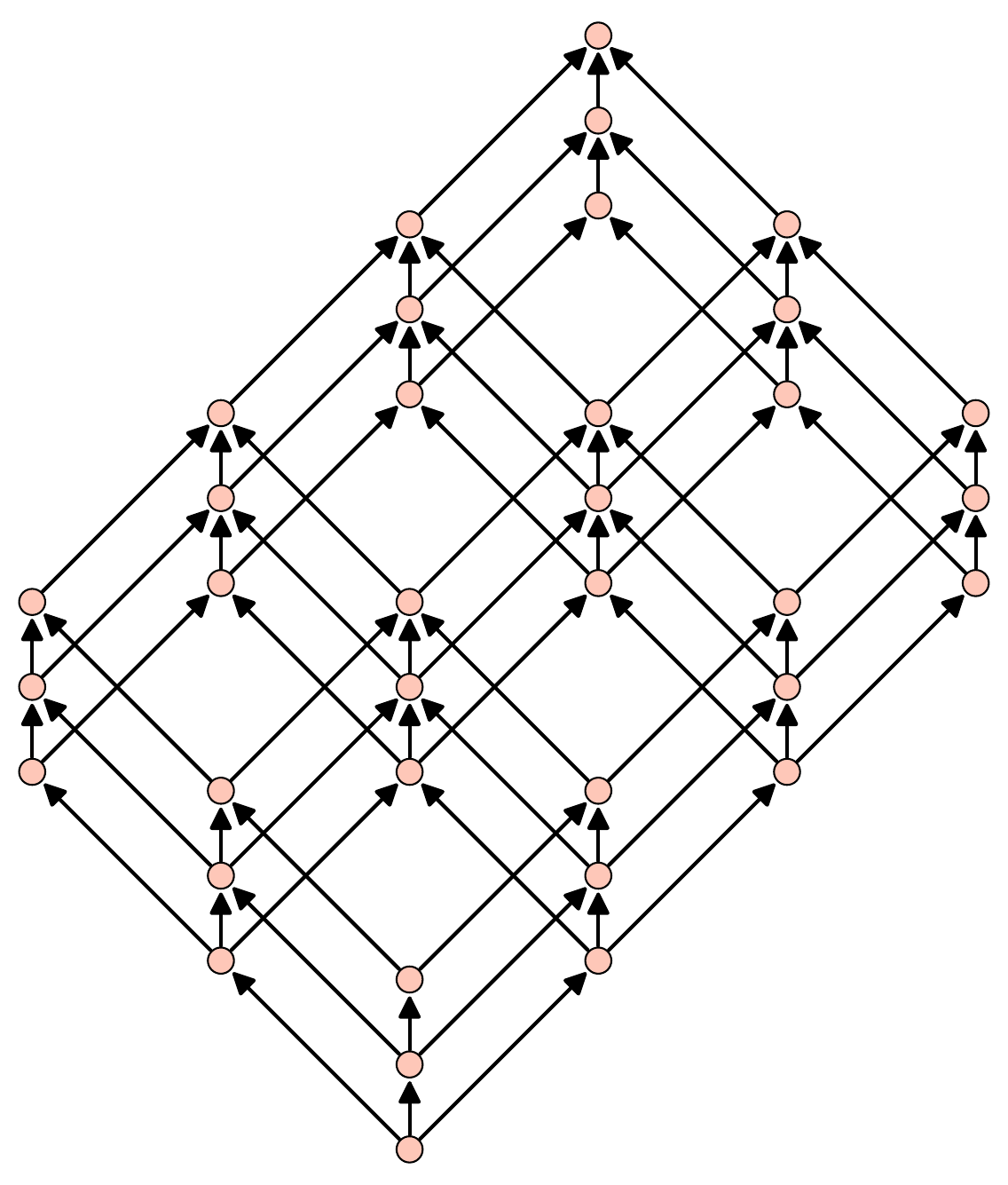}
	\end{minipage}
	
	\caption{\footnotesize{The Hasse diagram for the interval $[1 , s_1 s_0 s_2 s_1 s_0 s_2 s_0]$ in type $\tilde{A}_2$ is depicted on the left. This graph is spanned by the cubical lattice $\cC(2,2,3)$ on the right, and hence the Bruhat graph for this interval can be cubulated.}}%
	\label{fig:example-sublattice}
\end{figure}

After posting the first version of this paper, we learned that when $W$ is finite and $y = w_0$ is the longest element of $W$, cubulation of the Bruhat graph for $[1,w_0]=W$ by a cubical lattice parameterized by the exponents of $W$ is equivalent to~$W$ admitting a \emph{Lehmer code}. The Lehmer code for the finite symmetric group $W = S_n$ is a classical notion, while Lehmer codes for arbitrary finite Coxeter groups were introduced by Bolognini and Sentinelli in~\cite{Bolognini2025}, and further investigated by Sentinelli and Zatti in~\cite{SentinelliZatti}. Recent work of Gaetz and Gao~\cite{GaetzGao}, motivated by \cite{Gasharov,BilleyFanLosonczy,Billey}, extends the notion of Lehmer codes to arbitrary elements~$y$, and $[1,y]$ admitting a Lehmer code in this sense is exactly equivalent to the Bruhat graph for $[1,y]$ admitting a cubulation. These equivalences are explained in detail in Section~\ref{sec:cubulationPoincare}. Our notion of cubulation applies explicitly to all elements of arbitrary Coxeter groups, and introduces a graph-theoretic (hence visual) perspective.

We are motivated by the relationship between cubulations (or Lehmer codes) and triviality of Kazhdan--Lusztig polynomials~\cite{KL1}. For $W$ a finite Weyl group, these concepts also relate to rational smoothness of Schubert varieties (see, for example,~\cite{Bolognini2025,SentinelliZatti,GaetzGao}). We follow the normalization of the Kazhdan--Lusztig polynomials $\Pxy = \Pxy(q)$ from, for example, the reference~\cite{BjoernerBrenti}, and thus define $\Pxy$ to be \emph{trivial} if $\Pxy = 1$. Using the rank function of a cubical lattice, we calculate the Poincar\'e polynomial of the Bruhat interval~$[1,y]$ in the case that its Bruhat graph can be cubulated; see Proposition~\ref{prop:cubical-implies-quantum}. Combining this with results of Carrell--Peterson~\cite{Carrell} and Elias--Williamson~\cite{EliasWilliamson} yields:

\begin{restatable*}{theoremx}{MainThmA}%
\label{thm:Cubical=>Trivial}
  Let $(W,S)$ be an irreducible Coxeter system with $S$ finite, and let $y \in W$. If the Bruhat graph for $[1,y]$ can be cubulated, then $\Pxy = 1$ for all $x \leq y$.
\end{restatable*}

The majority of this paper investigates the converse to Theorem~\ref{thm:Cubical=>Trivial}. As explained further below, this converse is false in general; however, it does hold in certain contexts.  For example, in the special cases where either $y$ uses every simple reflection at most once, or $y$ is an arbitrary element in type $\tilde{A}_1$ or type $I_2(m)$ for $m \geq 3$, then the Bruhat graph for $[1,y]$ can be cubulated and $P_{x,y} = 1$ for all $x \leq y$; see Section \ref{sec:exs} for details. 

If $W$ is finite with longest element $w_0$, then $P_{x,w_0} = 1$ for all $x \in W$; see Exercise 7.14 in~\cite{Humphreys}. A Lehmer code was constructed in types $A_n$, $B_n/C_n$, $D_n$, and $H_3$  in~\cite[Sections 5.2--5.3]{Bolognini2025}. We give an alternative approach to the corresponding cubulations in types $A_n$ and $B_n/C_n$ in Appendix~\ref{sec:appendix}, where we make explicit use of a normal form forest of paths. (Our techniques do not extend to type~$D_n$, since there is no normal form forest of paths in this type.) From our inductive construction in type $A_n$ (respectively, $B_n/C_n$) with $S = \{ s_1, \dots, s_n\}$, it is immediate that the restriction of this cubulation to the standard parabolic subgroup of~$W$ of type $A_{n-1}$ (respectively, $B_{n-1}/C_{n-1}$) generated by $\{s_1, \dots, s_{n-1} \}$ is a cubulation of the Bruhat graph for this subgroup.

In addition, via a computation by the first author available at \cite{code}, we have verified the converse to Theorem~\ref{thm:Cubical=>Trivial} for all elements $y \in W$ in the following types: $A_3$, $A_4$, $B_3$, $B_4$, $D_4$, and $H_3$.
For the element $y = w_0$, we found using this same code that the graph $[1,w_0]$  \emph{cannot} be cubulated in types $E_6$, $F_4$, or $H_4$. In type $F_4$, this result was proven independently by Sentinelli and Zatti~\cite{SentinelliZatti} using similar methods.
We remark that the subsequent result \cite[Theorem 1.8]{GaetzGao} proves that for $(W,S)$ of finite type other than $E$, $F_4$,  or $H_4$, the converse to Theorem~\ref{thm:Cubical=>Trivial} holds for all $y \in W$, while in types $F_4$ and $H_4$ this converse holds if and only if $y \neq w_0$. In particular, \cite[Theorem 1.8]{GaetzGao} together with the work of \cite{Bolognini2025} answers a question about cubulation in type $D_n$ which we had posed in the first version of this paper.

We prove in Proposition~\ref{prop:embedding} that if the converse to Theorem~\ref{thm:Cubical=>Trivial} fails for some Coxeter system $(W,S)$, then it also fails for all Coxeter systems containing $(W,S)$ as a subsystem. We thus obtain the following ``poison subsystem'' result in Section \ref{sec:embedding}.

\begin{restatable*}{theoremx}{MainTheoremB}\label{thm:E7E8}
If $(W,S)$ has a subsystem of type $E_6$, $F_4$, or $H_4$, then there exists an element $y \in W$ such that $P_{x,y}=1$ for all $x \leq y$, but the Bruhat graph for $[1,y]$ cannot be cubulated. In particular, the converse to Theorem~\ref{thm:Cubical=>Trivial} does not hold in types $E_7$ or $E_8$.
\end{restatable*}

Note that Theorem~\ref{thm:E7E8} puts strong restrictions on the simply-laced Coxeter systems (finite or infinite) for which the converse to Theorem~\ref{thm:Cubical=>Trivial} might hold; namely, their Dynkin diagrams cannot contain any $E_6$ subdiagram. 

The remaining results of this paper concern infinite Coxeter systems. 
A Coxeter system $(W,S)$ is said to be \emph{minimal nonspherical} if $W$ is infinite, but every proper parabolic subgroup of~$W$ is finite. For example, any irreducible affine Coxeter system is minimal nonspherical. All other minimal nonspherical Coxeter systems are reflection groups of hyperbolic space; see Remark~\ref{rmk:minimal}. Utilizing a careful study of volume growth and Poincar\'e series in Coxeter groups, we obtain the following statement in Section~\ref{sec:CubulationGrowth}.

\begin{restatable*}{theoremx}{MainTheoremC}%
\label{thm:infinite} 
Let $(W,S)$ be a minimal nonspherical Coxeter system. If there are infinitely many distinct elements $y \in W$ such that the Bruhat graph $[1,y]$ can be cubulated, then $(W,S)$ is of type $\tilde{A}_n$ for some $n \geq 1$.
\end{restatable*}

Our final result provides one complete example of an infinite Coxeter system where the converse to Theorem~\ref{thm:Cubical=>Trivial} holds. We use the explicit formulas for Kazhdan--Lusztig polynomials in type $\tilde{A}_2$ provided by Libedinsky--Patimo~\cite{LibedinskyPatimo} and Burrull--Libedinsky--Plaza~\cite{BurrullLibedinskyPlaza} to give a constructive proof of the following statement, in Section \ref{sec:A2tilde}.

\begin{restatable*}{theoremx}{MainTheoremD}%
\label{thm:A2tilde} 
Suppose $(W,S)$ is of type $\tilde{A}_2$. Then for all (%
i.e.~the infinitely many) $y \in W$ such that $\Pxy = 1$ whenever $x \leq y$, the Bruhat graph for $[1,y]$ can be cubulated.  That is, the converse to Theorem~\ref{thm:Cubical=>Trivial} holds in type $\tilde{A}_2$.
\end{restatable*}

\subsection{Organization of the paper}

In Section~\ref{sec:preliminaries}, we recall key concepts for Coxeter systems, directed graphs, formal power series and growth in finitely generated groups.  Section \ref{sec:graphs} then introduces our main new concept, the cubulation of a directed graph. We relate cubulation to Lehmer codes and Poincar\'e polynomials, and establish Theorem~\ref{thm:Cubical=>Trivial}. 
Section \ref{sec:converse} investigates several special cases of the converse to Theorem~\ref{thm:Cubical=>Trivial} and discusses our computational results, as well as proving Theorem \ref{thm:E7E8}.  
Section \ref{sec:CubulationGrowth} proves Theorem \ref{thm:infinite}, and we provide a constructive proof of Theorem \ref{thm:A2tilde}  in Section \ref{sec:A2tilde}.

\subsection{Acknowledgements}

We thank Geordie Williamson for suggesting AB investigate Kazhdan--Lusztig polynomials for certain hyperbolic reflection groups, and for several helpful discussions.  We are grateful to the authors of~\cite{Bolognini2025,SentinelliZatti,GaetzGao} for making the connection between Lehmer codes and our notion of cubulation. We thank Haverford College for supporting a visit by AT in December 2022. Part of this work was completed while EM was a Member in the School of Mathematics at the Institute for Advanced Study, and we are grateful for those excellent working conditions.


\section{Preliminaries}\label{sec:preliminaries}

This section presents the relevant background from several different areas of mathematics.
Section \ref{sec:CoxeterBruhat} recalls basic notions related to Coxeter systems and their partial orders. Section~\ref{sec:directed} reviews directed graphs, including those associated to Coxeter systems. In Section~\ref{sec:powerSeries} we briefly review formal power series, and in Section~\ref{sec:poincare} we recall background on Poincar\'e series for Coxeter systems. Section~\ref{sec:growth} recalls the general theory of growth in finitely generated groups, and finally in Section~\ref{sec:growthCoxeter} we review growth in Coxeter groups. 

\subsection{Coxeter systems and partial orders}\label{sec:CoxeterBruhat}

In this section, we briefly review our notation for Coxeter systems and two natural partial orders they admit.

Throughout this work, $(W,S)$ is a Coxeter system with finite generating set $S = \{ s_1, \dots, s_n\}$. We say that $(W,S)$ is a \emph{finite} Coxeter system if the group~$W$ is finite. Denote by $[n] = \{1,\dots,n\}$ and by $\N = \{0, 1, 2, \dots \}$. %

We write $\ell$ for the word length on $W$ with respect to $S$. Given $x \in W$ such that $\ell(x) = k$, any product of the form $x = s_{i_1}s_{i_2}\cdots s_{i_k}$ where $s_{i_j} \in S$ is a \emph{reduced expression}. The \emph{support} of an element $x \in W$ is the set of elements of $S$ which appear in some (hence any) reduced expression for $x$. 
The empty word is the identity in $W$, which we denote by 1. 
We write $\leq_S$ for the (right) \emph{weak order} on $(W,S)$. This partial order is generated by the relations $x \leq_S xs$ for all $x \in W$ and $s \in S$ such that $\ell(xs) = \ell(x) + 1$; see, for example, \cite[Def.~3.1.1]{BjoernerBrenti}. Note that since $s \in S$, the condition $\ell(xs) = \ell(x)+1$ is equivalent to $\ell(xs) > \ell(x)$.

We write $\cT = \cT(W,S)$ for the set of \emph{reflections} in $(W,S)$; that is 
\[\cT = \{ xsx^{-1} \mid x \in W,\; s \in S \}.\] 
We write $\leq$ for the (right) \emph{Bruhat order} on $(W,S)$; this is sometimes referred to as \emph{strong order}, to contrast with the weak order on $(W,S)$. Recall that the Bruhat order is generated by $x \leq xt$ for all $x \in W$ and $t \in \cT$ such that $\ell(xt) > \ell(x)$; see, for example, \cite[Def.~2.1.1]{BjoernerBrenti}. We will also sometimes use the following characterization of Bruhat order via subwords: if $x, y \in W$, then $x \leq y$ if and only if for every reduced expression $s_{i_1} \cdots s_{i_k}$ for $y$, there exists a reduced expression for $x$ which is a (possibly non-consecutive) subword of $s_{i_1} \cdots s_{i_k}$; see, for example, \cite[Theorem 2.2.2]{BjoernerBrenti}. We write $[x,y]$ for intervals in Bruhat order. Note that $[x,y]$ is a graded poset, with rank function given by the word length $\ell$.

\subsection{Directed graphs}\label{sec:directed}

In this section, we review basic terminology about directed graphs and recall two important examples: the Cayley graph and the Bruhat graph for a Coxeter system $(W,S)$.

Let $\cD$ be a directed graph. We write $V(\cD)$ for the vertex set of $\cD$ and $E(\cD)$ for the edge set of $\cD$. All of the directed graphs $\cD$ appearing in this work will be \emph{simple}, meaning that all edges have distinct start- and end-vertices, and that for any two distinct vertices $u,v \in V(\cD)$, there is at most one edge whose endpoints are $\{u,v\}$. Hence, we regard $E(\cD)$ as a set of ordered pairs $(u,v) \in V(\cD) \times V(\cD)$ of distinct elements $u \neq v$ of $V(\cD)$.

We will often consider the following two kinds of subgraphs of directed graphs.

\begin{definition}[Induced and spanning subgraphs]\label{defn:subgraph} Let $\cD$ be a directed graph.
\begin{enumerate}
\item Let $V'$ be a subset of $V(\cD)$. The \emph{subgraph induced by~$V'$} is the subgraph of $\cD$ with vertex set $V'$ and edge set $\{ (u,v) \in E(\cD) \mid u, v \in V' \}$.
\item Let $\cD'$ be a subgraph of $\cD$. We say that $\cD'$ is a \emph{spanning subgraph} of $\cD$, or that $\cD'$ \emph{spans} $\cD$, if $V(\cD') = V(\cD)$. 
\end{enumerate}
\end{definition}

Note that if $\cD'$ spans $\cD$, then the edge set $E(\cD')$ will, in general, be a proper subset of $E(\cD)$. 

We will consider several different directed graphs in this work. The following two graphs correspond to the (right) weak order $\leq_S$ and the (right) Bruhat order $\leq$, respectively, which were defined in Section~\ref{sec:CoxeterBruhat}.

The \emph{(directed right) Cayley graph} $\cG = \cG(W,S)$ of $W$ is the directed graph with vertex set $V(\cG) = W$ and an edge from $x \in W$ to $xs \in W$, where $s \in S$, if and only if $\ell(xs) = \ell(x) + 1$. Equivalently, the edges of $\cG$ correspond to the covering relations in the (right) weak order $\leq_S$ on $W$. For any $x, y \in W$ with $x  \leq_S y$, we write $[x,y]_S$ for the subgraph of $\cG$ induced by the vertex set $\{ z \in W \mid x \leq_S z \leq_S y \}$. In other words, $[x,y]_S$ is the Hasse diagram for the partial order $\leq_S$ on the interval between $x$ and $y$.  The Cayley graph for type $A_2$ is seen in Figure \ref{fig:A2CayleyBruhat} by taking only the black edges, oriented upwards.

\begin{figure}[h]
 \resizebox{1.5in}{!}
 {
\begin{overpic}{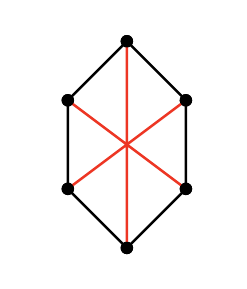}
\put(13,32){$s_1$}
\put(69,32){$s_2$}
\put(8,60){$s_1s_2$}
\put(67,60){$s_2s_1$}
\put(18,90){$s_1s_2s_1=s_2s_1s_2$}
\put(42,6){$e$}
\end{overpic}
}
\caption{\footnotesize{The Bruhat graph for type $A_2$, with all edges oriented upwards. The Hasse diagram for the Bruhat order consists of all edges except for the vertical red one, and the subgraph consisting of the black edges is the Cayley graph in type $A_2$.}}
\label{fig:A2CayleyBruhat}
\end{figure}

The \emph{Bruhat graph} $\cB = \cB(W,S)$ is the directed graph with vertex set $V(\cB) = W$ and an edge from $x \in W$ to $xt \in W$, for $t \in \cT$,  if and only if $\ell(xt) > \ell(x)$. Equivalently, the edges of $\cB$ correspond to the relations which generate the (right) Bruhat order $\leq$ on $W$. The Bruhat graph for type $A_2$ is shown in Figure \ref{fig:A2CayleyBruhat}. For any $x, y \in W$ with $x \leq y$, we write $[x,y]_\cB$ for the subgraph of $\cB$ induced by the vertex set $[x,y]$. By definition, the Hasse diagram for (the restriction of) the Bruhat order on $[x,y]$ spans the subgraph $[x,y]_\cB$. Since $S \subseteq \cT$, the subgraph $[x,y]_S$ of the directed Cayley graph $\cG$ spans this Hasse diagram, and hence also spans $[x,y]_\cB$.

\subsection{Formal power series}\label{sec:powerSeries}

We briefly recall basic definitions and a useful lemma.

Let $F(q) = \sum_{j=0}^\infty a_j q^j$ and $G(q) = \sum_{j=0}^\infty b_j q^j$ 
be (formal) power series, with $a_j, b_j \in \C$. For any $k \in \N$, write 
\[
F[k](q) = \sum_{j=0}^k a_j q^j
\]
for the Taylor polynomial consisting of all terms of $F(q)$ of degree $\leq k$.  
Recall that the (Cauchy) product of $F(q)$ and $G(q)$ is the power series
\[
F(q) G(q) = \sum_{j=0}^\infty c_j q^j \quad \mbox{ where } c_j = \sum_{k=0}^j a_k b_{j-k}.
\]

The next result is a straightforward consequence of these definitions.

\begin{lemma}\label{lem:power-series-product-agreement}
	Let $F(q)$, $G(q)$, and $H(q)$ be power series, and define \[ \Phi(q) = F(q)H(q) \quad \mbox{and} \quad \Gamma(q) = G(q)H(q).\]
	Then for all $k \in \N$, if $F[k](q) = G[k](q)$, we have $\Phi[k](q) = \Gamma[k](q)$. 
\end{lemma}

\subsection{Poincar\'e series for Coxeter systems}\label{sec:poincare}

In this section, we recall the definition and several useful results concerning Poincar\'e series for a Coxeter system $(W,S)$. 

\begin{definition}[Poincar\'e series] 
Given any subset~$A$ of $W$, the formal power series 
\[
A(q) = \sum_{a \in A} q^{\ell(a)}
\]
is called the \emph{Poincar\'e series} of $A$, or, if the set $A$ is finite, the \emph{Poincar\'e polynomial} of $A$. 
\end{definition}

We will consider two special cases. 
First, for any $y \in W$ and any $0 \leq j \leq \ell(y)$, define the nonnegative integer $c_j(y)$ by
$
c_j(y) =  \#\{ x \leq y \mid \ell(x) = j \}.
$
We will write $p_y(q)$ or sometimes just $p_y$ for the Poincar\'e polynomial of the (finite) Bruhat interval $A = [1,y] \subseteq W$. That is,	
\[
	p_y(q) = p_y = \sum_{j=0}^{\ell(y)} c_j(y) q^j.
\]	
We will also consider the Poincar\'e series with $A = W$, that is,
\[
W(q) = \sum_{x \in W} q^{\ell(x)}.
\]
In some of the literature, the Poincar\'e series $W(q)$ is referred to instead as the \emph{growth series} or the \emph{spherical growth series}. To avoid confusion with the volume growth series defined in Section~\ref{sec:growth} below, we avoid this terminology.

Recall that for any positive integer $n$, the \emph{$q$-analog of $n$} is the polynomial given by
\[
  [n]_q \coloneqq \frac{1-q^n}{1-q} = 1 + q + q^2 + \cdots + q^{n-1}.
\]
These polynomials can be used to state the following fundamental result. 

\begin{thm}[see, for example, Theorem~7.1.5 of \cite{BjoernerBrenti}]\label{thm:exponents} Let $(W,S)$ be an irreducible finite Coxeter system, with $S = \{ s_1, \dots, s_n \}$.  Then there are positive integers $e_1, \dots, e_n$ such that 
\[
  W(q) = \prod_{i=1}^n [e_i + 1]_q.
\]
Moreover, 
the order of the finite group $W$ is given by the product $\prod_{i=1}^n (e_i + 1)$, and the number of reflections in $W$ is given by the sum $\sum_{i=1}^n e_i$.
\end{thm}

\begin{definition}[Exponents]\label{defn:exponents} The positive integers $e_1,\dots,e_n$ given by Theorem~\ref{thm:exponents} are called the \emph{exponents} of the irreducible finite Coxeter system $(W,S)$. 
\end{definition}

Suppose now that $(W,S)$ is an irreducible affine Coxeter system of type~$\tilde{X}_n$, where $X \in \{A, B, C, D, E, F, G, H \}$ with appropriate restrictions on $n$. Then we will denote the $(n+1)$ elements of the affine generating set $S$ by $s_0,\dots,s_n$, and write $(W_0,S_0)$ for the associated spherical Coxeter system of type $X_n$, with generating set $S_0 = \{ s_1,\dots, s_n \} \subset S$.

\begin{thm}[Bott, see Theorem 7.1.10 of \cite{BjoernerBrenti}]\label{thm:Bott} Let $(W,S)$ be an irreducible affine Coxeter system with $|S| = n+1 \geq 2$,  and let $e_1,\dots,e_n$ be the exponents of the corresponding irreducible finite Coxeter system. Then the Poincar\'e series of $W$ is given by 
\[
W(q) = \prod_{i=1}^n \frac{[e_i + 1]_q}{1 - q^{e_i}}.
\]
\end{thm}

\begin{corollary}\label{cor:affinePoles} Let $(W,S)$ be an irreducible affine Coxeter system with $|S| = n+1 \geq 2$. The following are equivalent:
\begin{enumerate}
\item all poles of the Poincar\'e series $W(q)$ are at $q = 1$;
\item $(W,S)$ is of type $\tilde{A}_n$.
\end{enumerate}
\end{corollary}

\begin{proof} The exponents in finite type $A_n$ are
$
	\{ e_1,e_2,\ldots,e_n \} = \{ 1,2,\ldots,n \} = [n]
$
and thus the Poincar\'e series in type $\tilde{A}_n$ is $(1-q^{n+1})/(1-q)^{n+1}$. For all other types, we see from \cite[Table I, Appendix A1]{BjoernerBrenti} that $W(q)$ has a pole at a $k^{\mbox{\tiny{th}}}$ root of unity $e^{2\pi i/k}$ with $k \geq 3$, and hence $e^{2\pi i/k} \neq 1$. 
\end{proof}

\subsection{Growth in finitely generated groups}\label{sec:growth}

We now give background on volume growth for finitely generated groups, mostly following the exposition in~\cite[Chapter~6]{Loeh}. Throughout this section, $G$ is any finitely generated group, and $S \subseteq G$ is any finite generating set for $G$. Although we will be applying the theory of volume growth just to Coxeter systems $(W,S)$, where the group $W$ comes with a fixed (finite) generating set $S$, in order to present this material we need to work for now in this greater level of generality.

Define $S^{-1} = \{ s^{-1} \mid s \in S \}$.  (In a Coxeter system $(W,S)$, we have $S^{-1} = S$.) The \emph{word length function} on $G$ with respect to $S$ is the function $\ell_S: G \to \N$ given by, for any $g \in G$, 
\[
\ell_S(g) = \min \{ k \mid g = s_{i_1} \dots s_{i_k} \mbox{ where }s_{i_j} \in S \cup S^{-1} \mbox{ for }1 \leq j \leq k \}.
\]
The corresponding \emph{word metric} $d_S: G \times G \to \N$ is given by $d_S(g,h) = \ell_S(g^{-1}h)$  for any $g, h \in G$.
For any $k \in \N$, define
\[
B_{G,S}(k) = \{ g \in G \mid d_S(1,g) \leq k \} = \{ g \in G \mid \ell_S(g) \leq k \}
\]
to be the (closed) ball in $G$ of radius $k$ around the identity element $1$, with respect to the word metric $d_S$. We note that, since $S$ is finite, the ball $B_{G,S}(k)$ has finitely many elements for every $k \geq 0$. We can thus make the following definitions.

\begin{definition}[Volume growth series]\label{defn:growth} The \emph{volume growth function of $G$ with respect to $S$} is the map $\beta_{G,S}:\N \to \N$ given by
\[
\beta_{G,S}(k) = \# B_{G,S}(k).
\]
The \emph{volume growth series of $G$ with respect to $S$} is the formal power series
\[
	\Gamma_{G,S}(q)
	=
	\sum_{k=0}^\infty \beta_{G,S}(k) q^k.
\]
\end{definition}

Note that if the group $G$ is infinite, the function $\beta_{G,S}:\N \to \N$ is strictly increasing, meaning that for all $k \in \N$, we have $\beta_{G,S}(k) < \beta_{G,S}(k+1)$.

We next recall a partial order and equivalence relation on \emph{generalized growth functions}, which are just strictly increasing functions $\R_{\geq 0} \to \R_{\geq 0}$. By~\cite[Example 6.2.3]{Loeh}, if~$G$ is infinite then for any finite generating set $S$ for $G$, the volume growth function $\beta_{G,S}$ induces a generalized growth function, given by $r \mapsto \beta_{G,S}(\lceil r \rceil)$  for all $r \geq 0$. Let $\beta_1$ and $\beta_2$ be generalized growth functions. We say that $\beta_1$ is \emph{quasi-dominated} by $\beta_2$, denoted $\beta_1 \prec \beta_2$, if there exists $c \in \N$ so that for all $r \geq 0$,
\[
\beta_1(r) \leq c \beta_2(c r + c) + c.
\]
We then define $\beta_1$ and $\beta_2$ to be \emph{quasi-equivalent}, denoted $\beta_1 \sim \beta_2$, if $\beta_1 \prec \beta_2$ and $\beta_2 \prec \beta_1$.

Now let $G_1$ and $G_2$ be infinite, finitely generated groups, with finite generating sets $S_1$ and~$S_2$, respectively. We extend the  definitions in the previous paragraph to the volume growth functions $\beta_{G_i,S_i}:\N \to \N$ for $i = 1,2$, by considering the associated generalized growth functions $\R_{\geq 0} \to \R_{\geq 0}$. Explicitly, we have $\beta_{G_1,S_1} \prec \beta_{G_2,S_2}$ if and only if there exists $c \in \N$ so that for all $k \in \N$,
\[
\beta_{G_1,S_1}(k) \leq c \beta_{G_2,S_2}(c k + c) + c.
\]

\begin{lemma}[Section 6.2.1 of~\cite{Loeh}]\label{lem:growth} 
 The relation $\sim$ is an equivalence relation on generalized growth functions and hence on volume growth functions of infinite, finitely generated groups. Moreover, the relation of quasi-domination induces a partial order on the corresponding equivalence classes.
\end{lemma}

We now give a special case of Proposition 6.2.4 of~\cite{Loeh}.

\begin{prop}\label{prop:growthQI} Let $G$ be an infinite, finitely generated group, and let~$H$ be a finite-index subgroup of $G$. Then for any finite generating set $S$ for $G$ and any finite generating set $T$ for~$H$, the volume growth functions $\beta_{G,S}$ and $\beta_{H,T}$ are quasi-equivalent. In particular, for any two finite generating sets $S_1$ and $S_2$ for $G$, we have $\beta_{G,S_1} \sim \beta_{G,S_2}$.
\end{prop}

We can hence make the following important definition.

\begin{definition}[Polynomial growth] Let $G$ be an infinite, finitely generated group. We say that $G$ has \emph{polynomial growth} if for some (hence any) finite generating set $S$ for $G$, there is a positive integer $d$ such that $\beta_{G,S} \prec (r \mapsto r^d)$.
\end{definition}

We will use the following standard result, which follows from Proposition~\ref{prop:growthQI} and  the discussion of $\Z^n$ in \cite[Section 6.1]{Loeh}.

\begin{corollary}\label{cor:vZn} Let $G$ be an infinite, finitely generated group. If $G$ has a finite-index subgroup $H \cong \Z^n$, for some integer $n \in \N$, then $G$ has polynomial growth.
\end{corollary}

\subsection{Growth in Coxeter groups}\label{sec:growthCoxeter}

In this section we record some useful results for volume growth in the setting of Coxeter groups. 

Let $(W,S)$ be any Coxeter system (with $S$ finite). Let $W(q) = \sum_{x \in W} q^{\ell(x)} $ be the Poincar\'e series for $W$, as defined in Section~\ref{sec:poincare},  and let $\G_{W,S}(q)$ be the volume growth series for $W$ with respect to the generating set $S$, as given by Definition~\ref{defn:growth}. 
The series $W(q)$ and $\G_{W,S}(q)$ are related via the following easy observation (which, however, we could not find in the literature).

\begin{lemma}\label{lem:PoincareGrowth} 
Let $(W,S)$ be a Coxeter system. Then $
W(q) = (1 - q) \G_{W,S}(q)$.
\end{lemma}

Combining Lemma~\ref{lem:PoincareGrowth} with Theorem~\ref{thm:Bott} and Corollary~\ref{cor:affinePoles} above, we obtain the following.

\begin{corollary}\label{cor:affineGrowth} 
Let $(W,S)$ be an irreducible affine Coxeter system with $|S| = n+1 \geq 2$. Then the volume growth series $\G_{W,S}(q)$ is rational, and the following are equivalent:
\begin{enumerate}
\item all poles of $\G_{W,S}(q)$ are at $q = 1$;
\item $(W,S)$ is of type $\tilde{A}_n$.
\end{enumerate}
\end{corollary}

Next, we record the well-known fact that affine Coxeter groups have polynomial growth.

\begin{lemma}\label{lem:affinePoly} If $(W,S)$ is an irreducible affine Coxeter system, $W$ has polynomial growth.
\end{lemma}
\begin{proof} If $|S| = n+1$ then the translation subgroup of $W$ is free abelian of rank $n$, and has finite index in $W$. Hence by Corollary~\ref{cor:vZn}, the group $W$ has polynomial growth.
\end{proof}

Finally, we use a result of Terragni~\cite{Terragni} to give a short proof of the following converse to Lemma~\ref{lem:affinePoly}. We expect this statement is known to experts, and that it could be proved by other means, but could not find it written down explicitly. For $k \in \N$, denote by $c_k$ the coefficient of $q^k$ in the Poincar\'e series $W(q)$. The \emph{(exponential) growth rate} of $(W,S)$ is then defined by $\omega(W,S) = \lim \sup_k \sqrt[k]{c_k}$. 

\begin{thm}[Polynomial growth implies affine]\label{thm:polyAffine} Let $(W,S)$ be an irreducible Coxeter system. If $W$ is infinite and has polynomial growth, then $(W,S)$ is (irreducible) affine.
\end{thm}

\begin{proof} 
First, as noted in the introduction to~\cite{Terragni}, if $(W,S)$ has polynomial growth then $\omega(W,S) \leq 1$. But by Theorem (B) of~\cite{Terragni}, if $(W,S)$ is infinite non-affine then $\omega(W,S) \geq \tau$, where $\tau = 1.13\dots$ is a specified algebraic integer (which equals the growth rate for $(W,S)$ of type $E_{10}$). The result follows.
\end{proof}


\section{Cubulation}\label{sec:graphs}

We introduce our main new concept in \cref{sec:cubical}: this is the notion of a cubical lattice, which forms a poset graded by the $L^1$-norm on $\Z^N$.
\Cref{sec:cubulationPoincare} introduces our
terminology of cubulation, and relates the cubulation of subgraphs $[1,y]_\cB$ of the Bruhat graph to
\textit{Lehmer codes} (which we show equivalent to cubulations) and to the Poincar\'e polynomial $p_y$. We prove Theorem~\ref{thm:Cubical=>Trivial},
which relates cubulations to triviality of Kazhdan--Lusztig polynomials, in Section~\ref{sec:KL}.


\subsection{Cubical lattices}\label{sec:cubical}

We now introduce a special family of directed graphs which we call \emph{cubical lattices}, and discuss their structure as a graded poset. Recall that by $\N$ we mean the set of non-negative integers, i.e.~$\N = \{0,1,2,\ldots\}$.

\begin{definition}[Cubical lattice]\label{defn:cubical}
	Fix a positive integer $N$ and $k_1, \dots, k_N \in \N$. We define $\cC = \cC(k_1,\dots,k_N)$ to be the directed graph with vertex and edge sets
	\[ V(\cC) =	\bigl\{ (m_1,m_2,\ldots,m_N) \in \Z^N \bigm|	0 \leq m_i \leq k_i \text{ for each }i \in \{1,2,\ldots,N\} \bigr\} \]
  and
  \[ E(\cC)  = \bigl\{ (u,v) \in V(\cC) \times V(\cC) \bigm| v-u  = \vec{e}_i \ \text{for some}\ i \in \{1,2,\ldots,N\} \bigr\}, \]
where $\vec{e}_i$ denotes the $i^{\text{th}}$ standard basis vector in $\Z^N$. A directed graph $\cD$ is a \emph{cubical lattice} if $\cD$ is isomorphic to $\cC(k_1, \dots, k_N)$ for some $N \geq 1$ and $k_1,\dots,k_N \in \N$.
\end{definition}

If $k_1 = \dots = k_N = 0$, so that $\cC = \cC(k_1,\dots,k_N)$ is a single vertex (with no edges), we will sometimes write $\cC(0)$ instead of $\cC(0,\dots,0)$. We call $\cC(0)$ the \emph{trivial cubical lattice}. A cubical lattice $\cC(k_1,\dots,k_N)$ with at least one $k_i > 0$ is a \emph{nontrivial cubical lattice}. We note that we do not require any ordering on the parameters $k_1,\dots,k_N$, since it will sometimes be convenient to consider $k_1, \dots, k_N$ which are not (for example) weakly increasing. However, we will also use the following ``canonical form'' for cubical lattices.

\begin{lemma}\label{lem:permute} Any nontrivial cubical lattice $\cC(k_1,\dots,k_N)$ is naturally isomorphic to a cubical lattice $\cC(k_1',\dots,k'_{N'})$ where $k_j' > 0$ for $1 \leq j \leq N'$, and in addition, if desired, $k_1' \leq \dots \leq k'_{N'}$.
\end{lemma}
\begin{proof}
By permuting coordinates in $\Z^N$, we see that for any permutation $\sigma$ of $[N]$, the directed graph $\cC = \cC(k_1,\dots,k_N)$ is naturally isomorphic to $\cC(k_{\sigma(1)},\dots,k_{\sigma(N)})$. Hence, in particular, for some $0 \leq i < N$, the graph $\cC$ is isomorphic to $\cC' = \cC(k_1', \dots, k_N')$ where $k_1' = \dots = k_{i}' = 0$ and $k_j' > 0$ for all $i+1 \leq j \leq N$. Moreover, since $\cC$ and hence $\cC'$ is nontrivial, by dropping the first $i$ coordinates we see that $\cC'$ is naturally isomorphic to $\cC(k_{i+1}',\dots,k_N')$. The result then follows by relabeling the parameters, and, if desired, permuting them so that they are weakly increasing.
\end{proof}

An equivalent formulation of Definition~\ref{defn:cubical} is that the cubical lattice $\cC(k_1, \dots, k_N)$ is the Hasse diagram for the product of the subintervals $[0,k_i]$ of $\Z$, with each such subinterval a poset under the usual ordering. This  leads to the following result.

\begin{lemma}\label{lem:cubicalRank} 
Let $\cC = \cC(k_1, \dots, k_N)$ be a cubical lattice. Then the $L^1$-norm on $\Z^N$ given by
\[ \| (m_1, \dots, m_N) \|_1 = \sum_{i = 1}^N m_i \]
induces the structure of a graded poset on the vertex set $V(\cC)$. Moreover, this is the only possible rank function on $V(\cC)$.
 \end{lemma}
\begin{proof} The vertex set of any product of subintervals of $\N$, with the usual ordering on each subinterval, is a graded poset with rank function induced by the $L^1$-norm on $\Z^N$. Hence, viewing the cubical lattice $\cC=\cC(k_1,\dots,k_N)$ as the product of the subintervals $[0,k_i]$, the $L^1$-norm on $\Z^N$ induces the structure of a graded poset on $V(\cC)$. 

For uniqueness, we induct on $\sum_{i=1}^N k_i$. Observe that the result holds trivially when this sum equals $0$, equivalently $\cC = \cC(0)$ is trivial. Now suppose $\cC = \cC(k_1,\dots,k_N)$ is nontrivial. By Lemma~\ref{lem:permute}, we may assume, up to isomorphism of directed graphs (which will preserve any grading), that $1 \leq k_1 \leq k_2 \leq \cdots \leq k_N$. 

Define $V' \subseteq V(\cC)$ to be the set of vertices $(m_1, \dots, m_N) \in V(\cC)$ such that $0 \leq m_1 \leq k_1 - 1$ and $0 \leq m_i \leq k_i$ for $i \in \{2,\dots,N\}$, and let $\cC'$ be the subgraph of $\cC$ induced by $V'$. Then $\cC'$ is naturally isomorphic to the cubical lattice $\cC(k_1 - 1,k_{2},\dots,k_N)$. So by induction, there is a unique rank function on $V(\cC')$, namely that induced by the $L^1$-norm on~$\Z^N$.

Now the vertices of $\cC$ which are not in $\cC'$ are given by \[ V(\cC) \setminus V' = \bigl\{ (k_1,m_2,\dots,m_N) \bigm| 0 \leq m_i \leq k_i \text{ for }i \in \{2,\dots,N\} \bigr\}.\] Let $v = (k_1,m_2,\dots,m_N)$ be in $V(\cC) \setminus V'$, and define $v' = (k_1 - 1,m_2,\dots,m_n)$. Notice that $v' \in V'$, and that there is an edge in $\cC$ from $v'$ to $v$. Hence the only possible rank of $v$ which is compatible with the unique rank function on $V(\cC')$ is
\[
\| v' \|_1 + 1 = \left( (k_1 - 1) + \sum_{i=2}^N m_i \right) + 1 = k_1 + \sum_{i=2}^N m_i .
\]
But this sum equals $\| v \|_1$, which completes the proof.
\end{proof}

In order to further investigate the cubical lattice $\cC = \cC(k_1, \dots, k_N)$ as a graded poset, define \[ d_\cC = \sum_{i=0}^N k_i = \| (k_1,\dots,k_N) \|_1.\] Then by Lemma~\ref{lem:cubicalRank}, the maximum rank of any vertex of $\cC$ is $d_\cC$ (and the unique vertex of rank equal to $d_\cC$ is $(k_1,\dots,k_N)$). For $0 \leq j \leq d_\cC$, we now define natural numbers
\[
b_j(\cC) = \# \bigl\{ v \in V(\cC) \bigm| \| v \|_1 = j \bigr\}
\]
and the polynomial 
\[
g_\cC(q) = \sum_{j=0}^{d_\cC} b_j(\cC) \,q^j.
\]
That is, $b_j(\cC)$ is the number of vertices of the cubical lattice $\cC$ of rank exactly $j$, and $g_\cC(q)$ is the corresponding generating function. 

%

\begin{lemma}\label{lem:cubicalQuantum}
  Let $\cC = \cC(k_1, \dots, k_N)$ be a cubical lattice.
  Then 
  $
		g_\cC(q)=
    \prod_{i=1}^N [k_i + 1]_q.
  $
\end{lemma}
\begin{proof} For $0 \leq j \leq d_\cC = \sum_{i=1}^N k_i$, the coefficient of $q^j$ in the polynomial $g_\cC(q)$ is given by
\begin{align*}
	b_j(\cC)
	&=
	\#
	\{
		v \in V(\cC)
	\mid
		\|v\|_1 = j
	\}
	\\
	&=
	\#
	\left\{
		(m_1,m_2,\ldots,m_N) \in \Z^N
	\ \middle| \ 
		0 \leq m_i \leq k_i
		\text{ and }
		\sum_{i=1}^N m_i = j
	\right\}.
\end{align*}
This final expression for $b_j(\cC)$ is also clearly equal to the coefficient of $q^j$ in the product 
\[
	\bigl(1+q+q^2+ \cdots +q^{k_1}\bigr)
	\bigl(1+q+q^2+ \cdots +q^{k_2}\bigr)
	\cdots
	\bigl(1+q+q^2+ \cdots +q^{k_N}\bigr).
\]
The result follows.
\end{proof}

\subsection{Cubulations, Lehmer codes, and Poincar\'e polynomials}\label{sec:cubulationPoincare}

In this section we introduce our notion of cubulation, and relate cubulation to Lehmer codes as considered in~\cite{Bolognini2025,SentinelliZatti,GaetzGao}. We then use cubulations to determine certain Poincar\'e polynomials.

\begin{definition}
A directed graph $\cD$ \emph{can be cubulated} (alternatively, \emph{admits a cubulation}) if there is a cubical lattice~$\cC$ which is isomorphic to a spanning subgraph of $\cD$. In this case, we may say that $\cD$ is \emph{cubulated} by $\cC$, that $\cC$ \emph{cubulates} $\cD$, or that $\cC$ is a \emph{cubulation} of $\cD$. 
For convienence, we say that a Bruhat interval $[x,y]$ \emph{admits a cubulation} if $[x,y]_\cB$ does.
\end{definition}

See \cref{fig:example-sublattice} in the introduction for an example. We note that there may be more than one way of cubulating a given directed graph $\cD$. 

We now describe the relationship between cubulation and Lehmer codes. Let $(W,S)$ be a finite Coxeter system with exponents $e_1, \dots, e_n$ and longest element $w_0$.  
In their work~\cite{Bolognini2025}, Bolognini and Sentinelli define a \emph{Lehmer code} on $W$ to be a bijection $L$ from~$W$ to the product of intervals $\prod_{i=1}^n \{ 0, 1, \dots, e_i \}$ such that $L^{-1}$ is a poset homomorphism, with respect to the usual partial order on this product of intervals and the Bruhat order $\leq$ on~$W$.
It is clear that the existence of a Lehmer code on $W$ is equivalent to the Bruhat graph $[1,w_0]_\cB$ being cubulated by $\cC(e_1,e_2,\ldots,e_n)$: the Hasse diagram of the product of intervals $\prod_{i=1}^n \{ 0, 1, \dots, e_i \}$ is exactly the cubical lattice $\cC(e_1,e_2,\ldots,e_n)$, and $L^{-1}$ being a poset homomorphism onto $W$  is equivalent to this cubical lattice spanning $[1,w_0]_\cB$. 

More generally, Gaetz and Gao in Section 1.3 of~\cite{GaetzGao} define a \emph{Lehmer code} for~$[1,y]$, where $y \in W$ is arbitrary, to be an order-preserving bijection from a product of chains (in some poset) onto $[1,y]$. Their  motivation includes earlier consideration of such bijections in type~$A_n$~\cite{Gasharov} and for finite Weyl groups~\cite{BilleyFanLosonczy,Billey}. This notion of Lehmer code is easily seen to be equivalent to the Bruhat graph for $[1,y]$ being cubulated by $\cC(k_1,\dots,k_N)$, once each chain is identified with the appropriate subinterval $[0,k_i]$ of $\Z$. The scope of Coxeter groups $W$ considered in this definition of~\cite{GaetzGao} is a little unclear to us, since the focus of that work is on $W$ finite.

Let us return to our setting of an arbitrary Coxeter system $(W,S)$. We conclude this section by determining the Poincar\'e polynomial for an element $y \in W$ such that $[1,y]_\cB$ can be cubulated.  We will use the following result, including its final statement, to prove Theorem~\ref{thm:infinite} in Section~\ref{sec:CubulationGrowth}. Recall from Section~\ref{sec:cubical} that for any cubical lattice $\cC$, we denote by $g_\cC(q)$ the polynomial in which the coefficient of $q^j$ is the number of vertices of $\cC$ of rank exactly $j$.

\begin{prop}\label{prop:cubical-implies-quantum} 
Let $(W,S)$ be any Coxeter system, and let $y \in W$. 
	Suppose that the subgraph $[1,y]_\cB$ of the Bruhat graph $\cB$ is cubulated by $\cC = \cC(k_1,k_2,\ldots, k_N)$, where $N \in \N$ and  $k_1,\dots,k_N \in \N$.
	Then the Poincar\'e polynomial of $[1,y]_\cB$ is given by   $$
    p_y(q)=\prod_{i=1}^N [k_i + 1]_q.		
  $$

	Moreover, if $k_j \geq 1$ for all $1 \leq j \leq N$, then $N$ is the cardinality of the support of $y$.
\end{prop}

\begin{proof} Let $\varphi$ be an isomorphism from $\cC$ to a spanning subgraph of $[1,y]_\cB$. Then since $\varphi(\cC)$ spans $[1,y]_\cB$, we have $V(\varphi(\cC)) = V([1,y]_\cB) = [1,y]$. Now the graded poset $[1,y]$ and hence its spanning subgraph $\varphi(\cC)$ has rank function $\ell$, while by Lemma~\ref{lem:cubicalRank}, there is a unique rank function $\| \cdot \|_1$ on $V(\cC)$.  Thus, for any $v \in V(\cC)$ we have $\ell(\varphi(v)) = \| v \|_1$, and for all $0 \leq j \leq \ell(y)$, we have
\[
\# \bigl\{ x \leq y \bigm| \ell(x) = j \bigr\} = \# \bigl\{ v \in V(\cC) \bigm| \| v \|_1 = j \bigr\}.
\]
This is exactly saying that the coefficient of $q^j$ in the Poincar\'e polynomial $p_y(q)$ is equal to the coefficient of $q^j$ in the polynomial $g_\cC(q)$. In other words, $p_y(q) = g_\cC(q)$. The form of $p_y(q)$ now follows from Lemma~\ref{lem:cubicalQuantum}. 

For the final claim, observe that $s_i \in S$ is in the support of $y$ if and only if $s_i \in [1,y]$, and that the set of elements of $[1,y]$ of word length $1$ is exactly the set $S \cap [1,y]$. Thus, by Lemma~\ref{lem:cubicalRank} again, we have that $s_i \in S$ is in the support of $y$ if and only if the corresponding vertex~$v$ of $\cC(k_1,\dots,k_N)$ satisfies $\| v \|_1 = 1$. If each $k_j \geq 1$, then there are $N$ distinct vertices of $\cC(k_1,\dots,k_N)$ which have rank $1$, namely the standard basis vectors of $\Z^N$, and so the support of $y$ contains $N$ elements.
\end{proof}

\subsection{Cubulation and triviality of Kazhdan--Lusztig polynomials}\label{sec:KL}

In this section, we recall results of Elias and Williamson~\cite{EliasWilliamson} and Carrell and Peterson~\cite{Carrell} and combine these with \Cref{prop:cubical-implies-quantum} to prove Theorem~\ref{thm:Cubical=>Trivial}. 

Recall from the introduction that we follow the conventions of~\cite{BjoernerBrenti} for Kazhdan--Lusztig polynomials.  In these conventions, for $x, y \in W$ the Kazhdan--Lusztig polynomial $P_{x,y} = P_{x,y}(q)$ is a polynomial in $q$ with integer coefficients, so that if $x \not \leq y$, then $\Pxy(q) = 0$.
A fundamental result of Elias and Williamson~\cite{EliasWilliamson} says that the coefficients are in fact non-negative: 

\begin{thm}[Corollary~1.2 of~\cite{EliasWilliamson}] \label{thm:EW}
Let $(W,S)$ be any Coxeter system.  For any $x, y \in W,$ the Kazhdan--Lusztig polynomial satisfies $P_{x,y}(q) \in \Z_{\geq 0}[q]$.
\end{thm}

We say that $\Pxy$ is \emph{trivial} if $\Pxy(q) = 1$ is constant. A result of Carrell and Peterson~\cite{Carrell} provides several equivalent criteria for when Kazhdan--Lusztig polynomials are trivial.  (Note that in~\cite{Carrell} the left Bruhat order is used, but all results hold equally well for the right Bruhat order.) We will only state the equivalences from~\cite{Carrell} that we will use.  Recall that we denote by $p_y$ the Poincar\'e polynomial for the Bruhat interval $[1,y]$. 
A polynomial $f(q)$ of degree $m \geq 0$, given by
\[
	f(q) = a_0 + a_1 q + a_2 q^2 + \dots + a_m q^m,
\]
is said to be \emph{palindromic} if $a_j = a_{m-j}$ for all $0 \leq j \leq m$.

\begin{thm}[Theorem~B of~\cite{Carrell}]\label{thm:palindromic-iff-trivial}
	Let $(W,S)$ be any Coxeter system. Suppose $y \in W$ is such that for each $x \leq y$, the polynomial $\Pxy(q)$ has non-negative coefficients.
	Then the following are equivalent:
	\begin{enumerate}
		\item $\Pxy = 1$ for each $x \leq y$; and
		\item $p_y$ is palindromic.
	\end{enumerate}
\end{thm}

\begin{corollary}\label{cor:palindromic} 
Let $(W,S)$ be any Coxeter system, and let $y \in W$. Then $p_y$ is palindromic if and only if $\Pxy = 1$ for all $x \leq y$.
\end{corollary}

The next statement is elementary.

\begin{lemma}\label{lem:palindromic-closed-under-mult} 
	If $f(q)$ and $g(q)$ are palindromic polynomials, then $f(q)g(q)$ is palindromic. Hence in particular, any product of $q$-analogs is palindromic.
\end{lemma}

From this and \Cref{prop:cubical-implies-quantum}, we immediately obtain the following.

\begin{corollary}\label{cor:spanPal}
	Let $(W,S)$ be any Coxeter system, and let $y \in W$.  If $[1,y]_\cB$ can be cubulated, then $p_y$ is palindromic.
\end{corollary}

We now restate Theorem~\ref{thm:Cubical=>Trivial} from the introduction. Its proof is obtained by combining Corollaries~\ref{cor:spanPal} and~\ref{cor:palindromic} above.

\MainThmA

In the remainder of this work, we consider the converse to Theorem \ref{thm:Cubical=>Trivial}.


\section{Investigating the converse to Theorem~\ref{thm:Cubical=>Trivial}}\label{sec:converse}

In this section, we begin our investigations into the cases in which the converse to Theorem~\ref{thm:Cubical=>Trivial} holds. In Section~\ref{sec:exs}, we consider several special cases where it is straightforward to see that $\Pxy = 1$ for all $x \leq y$, and also easy to see that $[1,y]_\cB$ is spanned by a cubical lattice. We describe our computational results in Section~\ref{sec:computational}. Then in Section~\ref{sec:embedding} we show that if the converse to Theorem~\ref{thm:Cubical=>Trivial} fails for some Coxeter system $(W,S)$, then it fails for every Coxeter system which contains $(W,S)$ as a subsystem, and hence prove Theorem \ref{thm:E7E8}.


\subsection{Several special cases}\label{sec:exs}

This section considers cubulation of $[1,y]_\cB$ in the following special cases: when no simple reflection appears more than once in any reduced expression for $y$, when $(W,S)$ is a dihedral group, and when $y = w_0$ is the longest element in $(W,S)$ an irreducible finite Coxeter system of exceptional type.

\subsubsection{Standard parabolic Coxeter elements}\label{sec:boolean} 

An element $y \in W$ is called \emph{standard parabolic Coxeter} if each simple reflection in $S$ is used at most once in any (equivalently every) reduced expression for $y$.  As the terminology suggests, standard parabolic Coxeter elements are those that are Coxeter in some standard parabolic subgroup of $W$.  (Note that standard parabolic Coxeter elements also appear by other names; for example, they are called \emph{boolean} in some parts of the literature.)

Suppose $y \in W$ is a standard parabolic Coxeter element. Then, for any $x \leq y$, the interval $[x,y]$ is isomorphic as a poset to the Boolean lattice $B_{\ell(y) - \ell(x)}$.  Therefore, \cite[Cor.~6.8]{BrentiInv} says that $\Pxy = 1$. For any $k \in \N$, we write $\cC(1^k)$ for the cubical lattice $\cC(\,\underbrace{1,1,\dots, 1}_{k\ \text{times}}\,)$.

\begin{lemma}\label{lem:boolean}
Let $(W,S)$ be any Coxeter system. Suppose $y \in W$ is a standard parabolic Coxeter element with a reduced expression of the form $y = s_{i_1}\cdots s_{i_k}$.  Then $[1,y]_\cB$ is isomorphic to the cubical lattice $\cC(1^k)$, hence the converse to Theorem \ref{thm:Cubical=>Trivial} holds for this $y \in W$.
\end{lemma}

\begin{proof}
For $y = s_{i_1}\cdots s_{i_k}$ standard parabolic Coxeter, the poset $[1,y]$ is isomorphic to the Boolean lattice $B_{\ell(y)} = B_k$. Now clearly, the Hasse diagram for $B_k$ is isomorphic to the cubical lattice $\cC(1^k)$. Thus, the Hasse diagram for $[1,y]$ is isomorphic to $\cC(1^k)$.  

By the Strong Exchange Property, for any element $w = s_{j_1} \cdots s_{j_\ell}$ of the Bruhat interval $[1,y]$ and any reflection $t \in \cT$, we have $wt = s_{j_1} \cdots \widehat{s}_{j_m} \cdots s_{j_\ell}$ for some $m \in [\ell]$.  Since $y$ is standard parabolic Coxeter, so is $w$, and hence $\ell(wt) = \ell(w) -1$.  Therefore, in the case that $y$ is standard parabolic Coxeter, the Bruhat graph $[1,y]$ is equal to the Hasse diagram for $[1,y]$. We conclude that $[1,y]_\cB$ is isomorphic to $\cC(1^k)$, as required.
\end{proof}

For example, the left and middle of Figure~\ref{fig:booleanDihedral} depict the graphs $[1,s_i s_j]_\cB \cong \cC(1,1)$ and $[1,s_i s_js_k]_\cB \cong \cC(1,1,1)$, respectively, where $s_i$, $s_j$, and $s_k$ are three distinct simple generators.

\begin{figure}[ht!]
	\begin{subfigure}[m]{.3\linewidth}\centering
		\includegraphics[page=1]{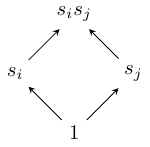}
	\end{subfigure}
	~
	\begin{subfigure}[m]{.3\linewidth}\centering
		\includegraphics[page=2]{fullgraph}
	\end{subfigure}
	~
	\begin{subfigure}[m]{.3\linewidth}\centering
		\includegraphics[page=3]{fullgraph}
	\end{subfigure}
	\caption{\footnotesize{From left to right, we depict a cubulation for standard parabolic Coxeter elements of lengths $2$ and $3$, and for dihedral elements of length $3$.}}\label{fig:booleanDihedral}
\end{figure}

Note the following consequence for elements of short word-length, since every element $y \in W$ such that $\ell(y) \leq 2$ is necessarily standard parabolic Coxeter.  By Exercise~7(a) in Chapter~5 of~\cite{BjoernerBrenti}, we have $\Pxy = 1$ for all $x \leq y$ with $\ell(y) \leq 2$.  Therefore, short word-length provides another context where the converse to Theorem \ref{thm:Cubical=>Trivial} holds.

\begin{corollary}\label{cor:short} 
Let $(W,S)$ be any Coxeter system, and suppose $y \in W$. If $\ell(y) \leq 2$, then $[1,y]_\cB$ is isomorphic to a cubical lattice. In particular, the converse to Theorem \ref{thm:Cubical=>Trivial} holds when $\ell(y) \leq 2$.
\end{corollary}


\subsubsection{Dihedral groups}\label{sec:dihedral}

If $(W,S)$ is a dihedral group (finite or infinite) and $y \in W$, then $\Pxy = 1$ for all $x \leq y$; see \cite[Section~7.12(a)]{Humphreys}. We show in the next result that the graph $[1,y]_\cB$ can be cubulated. We restrict to the case $\ell(y) \geq 3$ since $\ell(y) \leq 2$ is treated by Corollary~\ref{cor:short} above.

\begin{prop}\label{prop:dihedral} 
Let $(W,S)$ be a Coxeter system of type $I_2(m)$ for $m \geq 3$ or type $\tilde{A}_1$, and let $y$ be an element of $W$ such that $\ell(y) \geq 3$.
Then $[1,y]_\cB$ is cubulated by $\cC(1,\ell(y)-1)$. In particular, the converse to Theorem \ref{thm:Cubical=>Trivial} holds for dihedral groups.
\end{prop}

\begin{proof} Let $k = \ell(y)$. Since $(W,S)$ is dihedral, the generating set $S$ has exactly two elements. To simplify notation we put $S = \{ s,t\}$. Then, without loss of generality, $y$ has reduced expression the alternating word $sts \cdots$ containing $k$ letters. We induct on $k \geq 3$. 

Suppose first that $k = 3$; see the right of Figure~\ref{fig:booleanDihedral}. By Lemma~\ref{lem:boolean}, the graph $[1,st]_\cB$ is isomorphic to the square $\cC(1,1)$; the edges of this graph are the solid arrows on the right of Figure~\ref{fig:booleanDihedral}. Now the edge set of the graph $[1,sts]_\cB$ includes also the  edges $(t,ts)$, $(st,sts)$, and $(ts, ts(ststs)) = (ts, sts)$; these three edges are shown dashed on the right of Figure~\ref{fig:booleanDihedral}. Adding these three additional edges to the square $[1,st]_\cB$ results in a spanning subgraph of $[1,sts]_\cB$ which is isomorphic to $\cC(1,2)$. The proof of the inductive step is similar.
\end{proof}



\subsection{Computational results}\label{sec:computational}

Our exploration of whether the graph $[1,w_0]_\cB$ can be cubulated in the exceptional finite types $E$, $F$, and $H$  has largely been computational.
In particular, the first author has written a simple program (available on Zenodo~\cite{code} under the MIT License) to computationally check if such an interval can be cubulated.
We now provide a description of this program and its results. 

The operation of this program depends heavily on the Digraph and Coxeter libraries provided by SageMath~\cite{sage}.
Given the type of an irreducible finite Coxeter system $(W,S)$ as input, the program begins by using SageMath to generate the Bruhat poset $[1,w_0]$ together with its Hasse diagram.  The program then generates the rank sequence of the poset.  If the poset can be cubulated, then this rank sequence must exactly match the rank sequence for some cubical lattice.  Moreover, it is computable to find exactly what the dimensions of such a cubical lattice would
have to be, if it exists.  The program then generates a digraph (i.e.~a \verb!Digraph! object in SageMath) of such a cubical lattice and uses methods supplied by SageMath to determine whether this cubical lattice spans the Hasse diagram.
If there is such a cubulation, then it will use the isomorphism between the cubical lattice and some spanning subgraph of the Hasse diagram of $[1,w_0]$ to draw the Hasse diagram in such a way that the embedding is clear. 
In particular, \cref{fig:example-sublattice} was generated using a variant of this program.

At all steps of the above process, the program prints timestamped messages to a log file.
This log file also informs the user of the results of the program.
We note here that the program also supports affine Coxeter systems, although the user must provide an element $y\in W$, such that the program can then consider
the finite Bruhat interval $[1,y]$.
Variations of this program were crucial to the development of the proof provided in \cref{sec:A2tilde}.

Using this code, we have found that the graph $[1,w_0]_{\cB}$ cannot be cubulated in types $E_6$, $F_4$, and $H_4$. 
Investigation of other types has proven very computationally expensive using our present code; it did not yield results in type $D_6$ after $4$ months of compute time. Thus further computational investigations would require a different approach.


\subsection{An embedding result}\label{sec:embedding}

In this section, we show that if the converse to Theorem~\ref{thm:Cubical=>Trivial} fails for some Coxeter system $(W,S)$, then it fails for every Coxeter system which contains $(W,S)$ as a subsystem. We then restate and prove Theorem~\ref{thm:E7E8} from the introduction.

\begin{prop}\label{prop:embedding}
  Suppose $(W,S)$ is a Coxeter system so that, for some $y \in W$, we have $\Pxy = 1$ for all $x \leq y$, but $[1,y]_\cB$ cannot be cubulated. Let $(W',S')$ be any Coxeter system which contains $(W,S)$ as a subsystem. Denote by\/ $\leq'$ the Bruhat order on $(W',S')$, and let $\cB'$ be the Bruhat graph of $(W',S')$.

Write $P'_{x',y'}$ for the Kazhdan--Lusztig polynomial for $x',y' \in W'$. Then, regarding $y \in W$ as an element of $W'$, we have $P'_{x',y} = 1$ for all $x' \in W'$ with $x' \leq' y$, but $[1,y]_{\cB'}$ cannot be cubulated. In particular, the converse to Theorem~\ref{thm:Cubical=>Trivial} also fails for any Coxeter system which has $(W,S)$ as a subsystem.
\end{prop}

\begin{proof} This follows from the observation that for $x, y \in W$, both the Kazhdan--Lusztig polynomial $\Pxy$ and the graph $[1,y]_\cB$ are obtained using only elements of $W$ which have reduced expressions which are subwords of some reduced expression for $y$. But any reduced expression in $S'$ for $y$, regarding $y$ now as an element of $W'$, involves only letters in $S$.
\end{proof}

\MainTheoremB

\begin{proof}
Apply Proposition \ref{prop:embedding} to the fact from Section \ref{sec:computational} that $[1,w_0]_{\cB}$ cannot be cubulated in types $E_6$, $F_4$, or $H_4$.
\end{proof}


\section{Cubulations and growth in Coxeter groups}\label{sec:CubulationGrowth}

The goal of this section is to prove Theorem~\ref{thm:infinite} of the introduction. To prepare for this, in Section~\ref{sec:series} and Section~\ref{sec:PolynomialCriterion} we establish two technical results, concerning formal power series and polynomial growth, respectively.  The proof of Theorem~\ref{thm:infinite} is then carried out in Section~\ref{sec:growthCubulations}.

\subsection{Truncations of power series}\label{sec:series}

The proof of \cref{thm:infinite} will rely on the following technical lemma. It says that if the truncations of a power series are all polynomials of a very particular form, which arises in our argument, then the power series itself is either constant or a polynomial of a similar form.  

\begin{lemma}\label{lem:limit-of-quantum}
  Let $n \geq 1$, and let $(g_j(q))_{j=0}^\infty$ be a sequence of polynomials of the form
	\[
	g_j(q)
	=
	(1 - q^{a_{1,j}})
	(1 - q^{a_{2,j}})
	\cdots
	(1 - q^{a_{n,j}}),
	\]
	where for all $1 \leq i \leq n$ and all $j \geq 0$, the $a_{i,j}$ are integers satisfying \[ 1 \leq a_{1,j} \leq a_{2,j} \leq \cdots \leq a_{n,j}.\]
	Suppose that $F(q)$ is a power series such that for each $j \geq 0$, the Taylor polynomials $F[j](q)$ and $g_j[j](q)$ satisfy \[ F[j](q) = g_j[j](q).\] 
	Then either $F(q) = 1$, or there exists a positive integer $M$ and an integer $N$ with $1 \leq N \leq n$, such that 
	\[
	F(q) =
	(1 - q^{a_{1,M}})
	(1 - q^{a_{2,M}})
	\cdots
	(1 - q^{a_{N,M}})
	\]
	and for all $j \geq M$, we have $g_j(q) = g_M(q)$.
\end{lemma}

\begin{proof}
	We will prove the statement by induction on $n$. 
	Suppose $n = 1$. Then for each $j \geq 1$, we have $
	g_j(q) = 1 - q^{a_{1,j}}$.
	If $F(q) \neq 1$, then there must be some $m \in \N$ so that $g_m[m](q) \neq 1$. Then \[ g_m[m](q) = 1 - q^{a_{1,m}},\] and so in particular, $a_{1,m} \leq m$.
	Now for all $j \geq m$, we have $j \geq a_{1,m}$, so
	\[
	g_j[j](q) = 1 - q^{a_{1,m}}
	\]
	as well. Hence $g_j(q) = g_m(q)$ for every $j \geq m$. But then for every $j \geq m$, we also have
	\[
	F[j](q) = 1 - q^{a_{1,m}},
	\]
	and hence $F(q) = 1 - q^{a_{1,m}}$. Put $M = m$ and $N = n = 1$, and we have established the result for $n = 1$.
	
	For the inductive step, we again have that if $F(q) \neq 1$ then $g_m[m](q) \neq 1$ for some $m \in \N$. Now $a_{1,m} \leq \dots \leq a_{n,m}$, so if we expand out the product $g_m(q)$ we obtain
	\[
	g_m(q) = 1 - c_m q^{a_{1,m}} + \mbox{higher degree terms},
	\]
	where the coefficient $c_m \neq 0$ is the number of exponents $a_{1,m},\dots,a_{n,m}$ to be equal to $a_{1,m}$. We note that $a_{1,m} \leq m$. Hence for all $j \geq m$, since $j \geq a_{1,m}$ we have
	\[
	g_j[j](q) = 1 - c_m q^{a_{1,m}} + \mbox{higher degree terms}.
	\]
	Thus in particular, since $a_{1,j} \leq \dots \leq a_{n,j}$, we have $a_{1,j} = a_{1,m}$ for all $j \geq m$.
	So for all $j \geq m$ we have 
	\[
	\frac{g_j(q)}{1-q^{a_{1,m}}} =
	(1 - q^{a_{2,j}})
	\cdots
	(1 - q^{a_{n,j}}).
	\]
	We can thus define a sequence of polynomials $(h_j(q))_{j=0}^\infty$ by
	\[
	h_j(q) = \frac{g_{m+j}(q)}{1-q^{a_{1,m}}} =
	(1 - q^{a_{2,m+j}})
	\cdots
	(1 - q^{a_{n,m+j}}).
	\]
	
	Now as the quotient $1/(1-q^{a_{1,m}})$ 
	is itself a power series, we can view each $h_j(q)$ as the product of the polynomial $g_{m+j}(q)$ with this power series, and we can also define the product 
	\[
	G(q) = F(q) \left(
	\frac{1}{1-q^{a_{1,m}}} \right) = \frac{F(q)}{1-q^{a_{1,m}}}.
	\]
	Since $F[j](q) = g_j[j](q)$ for all $j \geq 0$, we can thus apply \cref{lem:power-series-product-agreement} to see that for all $j \geq 0$, 
	\begin{eqnarray*}
	G[m+j](q)  & = & \left( 
	\frac{F(q)}{1-q^{a_{1,m}}} 
	\right)
	[m+j](q) \\
	& = &
	\left( 
	\frac{g_{m+j}(q)}{1-q^{a_{1,m}}}
	\right)
	[m+j](q) \\ & = & h_j[m+j](q).
	\end{eqnarray*}
	Therefore $G[j](q) = h_j[j](q)$ for all $j \geq 0$.

	By inductive assumption, since each $h_j(q)$ is a product of $(n-1)$ factors, either $G(q) = 1$, or there is an $M' \geq 1$ and an integer $N'$ with $1 \leq N' -1 \leq n-1$, such that $G(q)$ is the product of $(N' - 1)$ factors as follows:
	\[
	G(q) 
	=
	(1 - q^{a_{2,M'}})
	\cdots
	(1 - q^{a_{N',M'}}),
	\]
	and for all $j \geq M'$, we have $h_j(q) = h_{M'}(q)$. Recall also from above that $a_{1,j} = a_{1,m}$ for each $j \geq m$. 
	
	If $G(q) = 1$, then $F(q) = 1 - q^{a_{1,m}}$ and $g_j(q) = g_m(q)$ for all $j \geq m$, and we are done with $M = m$ and $N = 1 \leq n$. Now assume $G(q) \neq 1$. Then we have in particular that $a_{i,m+j} = a_{i,m+M'}$ for all $2 \leq i \leq N'$ and all $j \geq M'$.  Put $M = m + M'$. Then $a_{1,m} = a_{1,M}$ and $a_{i,M'} = a_{i,M}$ for all $2 \leq i \leq n$, and so $F(q) = (1 - q^{a_{1,m}})G(q)$ is the product of $1 + (N'-1) = N'$ factors as follows:
	\[
	F(q)
	=
	(1-q^{a_{1,M}})
	(1 - q^{a_{2,M}})
	\cdots
	(1 - q^{a_{N',M}}),
	\]
	and for each $j \geq M$, we have $g_j(q) = g_M(q)$. Letting $N = N'$, this completes the proof of the inductive step.
\end{proof}

\subsection{A criterion for polynomial growth}\label{sec:PolynomialCriterion}

In this section we establish the following specialized statement, which we will need for our proof of Theorem~\ref{thm:infinite}. This says that if the volume growth series is a rational function of a very particular form, then the group $G$ has polynomial growth.

\begin{lemma}\label{lem:polyGrowth} Let $G$ be an infinite, finitely generated group with finite generating set $S$. Suppose that for some polynomial $f(q)$ with non-negative integer coefficients and some positive integer $m$, the volume growth series $\G_{G,S}(q)$ is given by the rational function
\[
\G_{G,S}(q) = \frac{f(q)}{(1-q)^m}.
\] Then $G$ has polynomial growth. 
\end{lemma}

\begin{proof} 
	We have 
	$\frac{1}{1-q} = \sum_{j=0}^\infty q^j$, and a straightforward induction shows that for all $m \geq 2$,
	\[
	\frac{1}{(1-q)^{m}}  = \frac{1}{(m-1)} \sum_{k=0}^\infty (k+m-1)(k+m - 2)\cdots(k+1)q^{k}.
	\]
	Put $b_{1,k} = 1$ for all $k \in \N$, and for all $m \geq 2$ and $k \in \N$, define
	\[
	b_{m,k} = \frac{(k+m-1)(k+m - 2)\cdots(k+1)}{(m-1)}. 
	\]
	Note that, since the numerator here contains $m-1$ consecutive positive integers, each $b_{m,k}$ is a positive integer. Also observe that for all $m \geq 1$ and all $k \in \N$, we have $b_{m,k} \leq b_{m,k+1}$.
	That is, for all $m \geq 1$, we have
	$
	\frac{1}{(1-q)^m} = \sum_{k=0}^\infty b_{m,k} q^k 
	$
	where the coefficients $b_{m,k}$ form a weakly increasing sequence of positive integers.
	
	Now write $f(q) = a_0 + a_1 z + \dots + a_M q^M$, where the $a_j$ are non-negative integers (and $a_M \neq 0$). Then the coefficient of $q^k$ in the volume growth series $\G_{G,S}(q) = f(q)/(1-q)^{m}$ is given by 
	\[ \beta_{G,S}(k)= \sum_{j=0}^k a_j b_{m,k-j}.\]
	Let $A$ be the positive integer $A = \max_{0 \leq j \leq M} a_j$. Then since the sequence $(b_{m,k})_{k=0}^\infty$ is weakly increasing, we have that for all $m \geq 1$ and all $k \in\N$,
	\[
	\beta_{G,S}(k) \leq A \sum_{j=0}^k b_{m,k-j} \leq A(k+1) b_{m,k}.
	\]
	If $m = 1$ then since $b_{1,k} = 1$ this implies \[ \beta_{G,S}(k) \leq A(k+1),\] while for $m \geq 2$ we get
	\[
	\beta_{G,S}(k) \leq \frac{A(k+1)(k+m-1)(k+m - 2)\cdots(k+1)}{(m-1)}. 
	\]
	Thus for all $m \geq 1$, we obtain that $\beta_{G,S}(k)$ is bounded above by a polynomial in $k$ (of degree~$m$). Therefore $\beta_{G,S}$ is quasi-dominated by a monomial (of degree $m$), and so $G$ has polynomial growth, as required.
\end{proof}

\subsection{Cubulations, growth, and type $\tilde{A}_n$}\label{sec:growthCubulations}

We now prove Theorem~\ref{thm:infinite}. For this, we will establish two key results, Propositions~\ref{prop:bruhat-volume} and~\ref{prop:growth-cubical}, and then combine Proposition~\ref{prop:growth-cubical} with statements from 
\cref{sec:preliminaries}.

Our results in this section concern the following class of Coxeter systems.

\begin{definition} Let $(W,S)$ be a Coxeter system with $S = \{ s_i \mid i \in [n] \}$. We say that $(W,S)$ is \emph{minimal nonspherical} if $W$ is an infinite group, but for each $i \in [n]$, the standard parabolic subgroup generated by $S \setminus \{ s_i \}$ is finite.
\end{definition}

Equivalently, $(W,S)$ is minimal nonspherical if every proper parabolic subgroup of $W$ is finite. Note that if $(W,S)$ is minimal nonspherical, then $|S|=n \geq 2$. 

\begin{rmk}\label{rmk:minimal} If $(W,S)$ is minimal nonspherical, then as discussed in Section~6.9 and Example~14.2.3 of \cite{Davis}, the generating set $S$ is the set of reflections in the faces of a compact simplex in either Euclidean or hyperbolic space. Hence $(W,S)$ is either irreducible affine, or irreducible hyperbolic as classified by Lann\'er; for the latter, see~\cite[Table~6.2]{Davis}.
\end{rmk}

We now establish Proposition~\ref{prop:bruhat-volume}. This says that in minimal nonspherical systems, the ball $B_{W,S}(k)$ is  contained in the Bruhat interval $[1,y]$ for all long enough $y$ (depending on $k$). 

\begin{prop}\label{prop:bruhat-volume}
	Let $(W,S)$ be a minimal nonspherical Coxeter system. Then there is an explicit constant $L = L(W,S) \geq 2$ such that for all $k \in \N$ and all $y \in W$ with $\ell(y) \geq k L$, we have
	$
	B_{W,S}(k) \subseteq [1,y].
	$
\end{prop}

\begin{proof}
	Since $(W,S)$ is minimal nonspherical, for each $i \in [n]$ the standard parabolic subgroup $W_{S \setminus \{ s_i \}}$ of $W$ is finite. Hence for each $i \in [n]$, we may define the positive integer $\ell_i$ to be the length of the longest element of $W_{S \setminus \{ s_i \}}$.
	We then define $L=L(W,S) \in \N$ by
	\[
	L = 1 + \max_{i \in [n]} \ell_i \geq 2.
	\]
	We note that if $y \in W$ is such that $\ell(y) \geq L$, then the support of $y$ must equal $S$. Otherwise,~$y$ would be contained in some proper standard parabolic subgroup of $W$, which would in turn imply $\ell(y) \leq L - 1$.
	
We now fix $k \in \N$, and let $y \in W$ be any element such that $\ell(y) \geq k L$. Then there is reduced expression for $y$ given by $y = y_1 \dots y_k,$
where for each $1 \leq j \leq k$, the subword $y_j$ is reduced, and $\ell(y_j) \geq L$. Hence each $y_j$ has support equal to $S$.
	
To complete the proof, let $x \in B_{W,S}(k)$. Then there is a reduced expression for $x$ given by
$x = s_{i_1} \dots s_{i_m},$ where $s_{i_j} \in S$ and $m \leq k$.  For all $1 \leq j \leq m$, since $y_j$ has support the entire set $S$, the letter $s_{i_j} \in S$ is a (proper) subword of the reduced expression $y_j$. Hence the reduced expression $s_{i_1} \dots s_{i_m}$ for $x$ is a subword of the reduced expression $y_1 \dots y_k$ for $y$.	Thus $x \leq y$ in Bruhat order, and so $B_{W,S}(k)$ is contained in the Bruhat interval $[1,y]$ as desired.
\end{proof}

We will use Proposition~\ref{prop:bruhat-volume} to prove our second key result, Proposition~\ref{prop:growth-cubical}, which describes the volume growth series for minimal nonspherical Coxeter systems $(W,S)$ in which $[1,y]_\cB$ can be cubulated for infinitely many distinct $y \in W$. The proof of Proposition~\ref{prop:growth-cubical} also makes essential use of Proposition~\ref{prop:cubical-implies-quantum}, which describes the Poincar\'e polynomial $p_y$ when $[1,y]_\cB$ can be cubulated, and the technical result established in Section~\ref{sec:series}.

\begin{prop}\label{prop:growth-cubical}
	Let $(W,S)$ be a minimal nonspherical Coxeter system with $|S| = n \geq 2$. 
	If there are infinitely many distinct $y \in W$ such that the Bruhat graph $[1,y]_\cB$ can be cubulated, then either 
	\[
	\G_{W,S}(q) = \frac{1}{(1-q)^{n+1}}
	\]
	or there is an integer $N$ with $1 \leq N \leq n$, and  integers $2 \leq a_1 \leq a_2 \leq \dots \leq a_N$, such that 
	\[
	\Gamma_{W,S}(q)= 
	\left(
	\frac{1 - q^{a_1}}{1-q}
	\right)
	\left(
	\frac{1 - q^{a_2}}{1-q}
	\right)
	\dots
	\left(
	\frac{1 - q^{a_N}}{1-q}
	\right)
	\frac{1}{(1-q)^{n+1-N}}.
	\]
\end{prop}

\begin{proof}
	Let $L = L(W,S) \geq 2$ be the constant from \cref{prop:bruhat-volume}. Set $y_0 = 1$.
	Then choose an infinite sequence $(y_j)_{j=0}^\infty$ of elements of $W$ such that:
	\begin{itemize}
		\item for each $j \geq 0$, the graph $[1,y_j]_\cB$ can be cubulated; and
		\item for each $j \geq 1$, we have $\ell(y_j) \geq L j$.
	\end{itemize}
Thus $B_{W,S}(0) = \{ 1\}$ and $[1,y_0] = \{1\}$, and by \cref{prop:bruhat-volume}, we have that for all $j \geq 1$, the ball $B_{W,S}(j)$ is contained in $[1,y_j]$. Hence, in particular, for all $j \geq 1$ the element~$y_j$ has support equal to $S$.
	
	To simplify notation we now write $p_j(q)$ for the Poincar\'e polynomial $p_{y_j}(q)$.
We note that each $p_j(q)$ has degree equal to $\ell(y_j)$, hence $\deg p_j(q) \geq L j \geq j+1$ for all $j \geq 1$. 
 As $B_{W,S}(j)$ is a subset of $[1,y_j]$, we thus obtain that for all $j \geq 0$, \[ W[j](q) = p_j[j](q), \]
	where we recall 
  that $W(q) = (1-q)\G_S(q)$ is the Poincar\'e series for $W$. 
	
	By assumption, for each $j \geq 1$, the graph $[1,y_j]_\cB$ can be cubulated by, say, the (nontrivial) cubical lattice $\cC(k_{1,j}, k_{2,j}, \dots, k_{N_j,j})$. We may assume $1 \leq k_{1,j} \leq k_{2,j} \leq \cdots \leq k_{N_j,j}$.
	Then by \cref{prop:cubical-implies-quantum}, for all $j \geq 1$ the Poincar\'e polynomial $p_j(q)$ is the product of 
  $q$-analogs
	\[
	p_j(q) = 
  \prod_{i=1}^{N_j} [k_{1,j} + 1]_q
	\]
	where, since each $k_{i,j} \geq 1$, in fact each $N_j = n = |S|$ is the cardinality of the support of $y_j$.  We hence simplify notation by putting $a_{i,j} = k_{i,j} + 1$ for all $1 \leq i \leq n$ and $j \geq 1$. Then for all $j \geq 1$, we have $2 \leq a_{1,j} \leq a_{2,j} \leq \dots \leq a_{n,j}$ and 
\[
	p_j(q) =  \prod_{i=1}^n \left( \frac{q^{a_{i,j}} - 1}{q-1} \right) = \prod_{i=1}^n \left( \frac{1 - q^{a_{i,j}} }{1-q} \right).
	\]
	
	We now define a sequence of polynomials $(g_j(q))_{j=0}^\infty$ by
	\[
		g_j(q) = (1-q)^{n} p_j(q),	
	\]
	and define a power series $F(q)$ by 
	\[
		F(q) = (1-q)^{n}W(q) = (1-q)^{n+1}\G_{W,S}(q).
	\]
	By \cref{lem:power-series-product-agreement}, since $W[j](q) = p_j[j](q)$, we have $F[j](q) = g_j[j](q)$ for each $j \geq 0$. 
	
	If $F(q) = 1$ then it is immediate that $\displaystyle \G_{W,S}(q) = \frac{1}{(1-q)^{n+1}}$. Otherwise, by \cref{lem:limit-of-quantum},
	\[
	F(q)
	=
	(1 - q^{a_{1,M}})
	(1 - q^{a_{2,M}})
	\cdots
	(1 - q^{a_{N,M}})
	\]
	for some $M \geq 1$ and some integer $N$ with $1 \leq N \leq n$.
	Upon dividing this polynomial expression for $F(q)$ through by $(1-q)^{n+1}$, and putting $a_i = a_{i,M} \geq 2$ for $1 \leq i \leq N$, we obtain the desired expression for the volume growth series $\G_{W,S}(q)$.	
\end{proof}

\begin{corollary}\label{cor:cubicalPoly}	Let $(W,S)$ be a minimal nonspherical Coxeter system. 
	If there are infinitely many distinct $y \in W$ such that $[1,y]_\cB$ can be cubulated, then $W$ has polynomial growth.
\end{corollary}
\begin{proof} Apply Lemma~\ref{lem:polyGrowth} to the form of the volume growth series $\G_{W,S}(q)$ established in  Proposition~\ref{prop:growth-cubical}, with $f(q)$ the product of 
  $q$-analogs
$
f(q) = 
\prod_{i=1}^N [a_i]_q
$, 
and $m = n+1 - N \geq 1$.
\end{proof}

We now complete the proof of Theorem~\ref{thm:infinite} by combinining Proposition~\ref{prop:growth-cubical} and Corollary \ref{cor:cubicalPoly} with results from 
\cref{sec:preliminaries}.

\MainTheoremC

\begin{proof}
Let $(W,S)$ be a minimal nonspherical Coxeter system, and suppose that there are infinitely many distinct $y \in W$ such that $[1,y]_\cB$ can be cubulated.  By Theorem~\ref{thm:polyAffine} and Corollary~\ref{cor:cubicalPoly}, since $W$ has polynomial growth,  the Coxeter system $(W,S)$ is irreducible affine. 
If $|S| = n +1 \geq 2$, then from Proposition~\ref{prop:growth-cubical} with $n$ replaced by $n+1$ in its statement, we see that all poles of the volume growth series $\G_{W,S}(q)$ are at $q = 1$. Hence, by Corollary~\ref{cor:affineGrowth}, we have that $(W,S)$ is of type $\tilde{A}_n$ for some $n \geq 1$. This completes the proof of Theorem~\ref{thm:infinite}.
\end{proof}

\begin{rmk}\label{rmk:hyp} We now sketch an alternative approach to part of the proof of Theorem~\ref{thm:infinite}. We will freely use some standard concepts and arguments from geometric group theory (see~\cite{Loeh}, as well as the reference~\cite{BridsonHaefliger}). Suppose $(W,S)$ is minimal nonspherical. Then by Remark~\ref{rmk:minimal}, either $(W,S)$ is irreducible affine, or the group $W$ acts properly discontinuously and cocompactly by isometries on $n$-dimensional hyperbolic space $\mathbb{H}^{n}$ (where $|S| = n+1$). In the latter case, $W$ is quasi-isometric to $\mathbb{H}^n$, hence is a Gromov-hyperbolic group, and therefore has exponential growth. Thus for $(W,S)$ minimal nonspherical, once it is known that~$W$ has polynomial growth (or even just subexponential growth), we can deduce that $(W,S)$ is irreducible affine without using Theorem~\ref{thm:polyAffine}.
\end{rmk}


\section{Construction of cubulations in type \texorpdfstring{$\tilde{A}_2$}{\~A2}}\label{sec:A2tilde}

We conclude by giving a constructive proof of Theorem~\ref{thm:A2tilde}, which concerns  $(W,S)$ of type $\tilde{A}_2$. In Section~\ref{sec:reduction}, we show that it suffices to consider one infinite family $\{ y_m \}_{m \in \N}$ of elements of~$W$, together with certain standard parabolic Coxeter elements, which are already addressed by Lemma \ref{lem:boolean}. We construct a cubulation of each graph $[1,y_m]_\cB$ in Section~\ref{sec:infiniteA2tilde}.

\subsection{Reduction}\label{sec:reduction}

In this section we carry out a reduction for the proof of Theorem~\ref{thm:A2tilde}, using results from Libedinsky--Patimo~\cite{LibedinskyPatimo} and Burrull--Libedinsky--Plaza~\cite{BurrullLibedinskyPlaza}. 

In order to obtain this reduction, we first record another condition equivalent to $\Pxy = 1$ for all $x \leq y$, complementing Theorem \ref{thm:palindromic-iff-trivial}. Following the notation of~\cite{LibedinskyPatimo}, and using the same normalizations as in that work, the Hecke algebra of $W$ is a $\Z[v,v^{-1}]$ module with two distinguished bases: the standard basis $\{ \Hgen_y\}_{y \in W}$, and the canonical (or Kazhdan--Lusztig) basis $\{ \KLbasis_y \}_{y \in W}$. These bases are related via the equation
\begin{equation}\label{eq:changeBasis}
\KLbasis_y = \sum_{x \leq y} h_{x,y} \Hgen_y
\end{equation}
where the $h_{x,y} = h_{x,y}(v)$ are the Kazhdan--Lusztig polynomials in Soergel's normalization~\cite{Soergel97}. As in~\cite{LibedinskyPatimo}, for $y \in W$ put 
\[
\Ngen_y = 
	\sum_{x\leq y} v^{\ell(y)-\ell(x)} \Hgen_y.
\]

We now relate these notions to triviality of the Kazhdan--Lusztig polynomials $\Pxy = \Pxy(q)$ as in Section~\ref{sec:KL}.

\begin{lemma}\label{lem:KLgenNgen} 
Let $(W,S)$ be any Coxeter system, and let $y \in W$. The following are equivalent:
\begin{enumerate}
\item $\KLbasis_y = \Ngen_y$;
\item for all $x \leq y$, we have $h_{x,y} = v^{\ell(y) - \ell(x)}$; and
\item for all $x \leq y$, we have $\Pxy = 1$.
\end{enumerate}
\end{lemma}
\begin{proof} 
The equivalence of (1) and (2) is due to $\{ \Hgen_y \}_{y \in W}$ being a basis for the Hecke algebra, Equation~\eqref{eq:changeBasis}, and the definition of $\Ngen_y$. The equivalence of (2) and (3) is then obtained by changing the normalization of Kazhdan--Lusztig polynomials. Specifically, making the identification $q = v^{-2}$ we get that $\Pxy = v^{\ell(x) - \ell(y)}h_{x,y}$, and hence $\Pxy = 1$ exactly when $h_{x,y} = v^{\ell(y) - \ell(x)}$.
\end{proof}

For $y,y' \in W$, write $y \sim y'$ if there is a diagram automorphism $\phi$ of $(W,S)$ such that $\phi(y) = y'$. The next statement, which applies Theorems~\ref{thm:EW} and~\ref{thm:palindromic-iff-trivial}, is similar to Proposition 2.14 of~\cite{BurrullLibedinskyPlaza}.

\begin{corollary}\label{cor:diagram} 
Let $(W,S)$ be any Coxeter system, and let $y \in W$.The following are equivalent:
\begin{enumerate}
\item $\KLbasis_y = \Ngen_y$; 
\item for all $x \leq y$, we have $\Pxy = 1$; and
\item for all $y' \in W$ such that $y \sim y'$, and for all $x \leq y'$, we have $P_{x,y'} = 1$.
\end{enumerate}
\end{corollary}
\begin{proof} Observe that both word length and the set of reflections in $W$ are invariant under diagram automorphisms. Hence if $y \sim y'$, the posets $[1,y]$ and $[1,y']$ are isomorphic, and so in particular, the Poincar\'e polynomials $p_y$ and $p_{y'}$ are identical. Therefore utilizing both implications in Corollary~\ref{cor:palindromic} (which follows from Theorems~\ref{thm:EW} and~\ref{thm:palindromic-iff-trivial}), we have $\Pxy = 1$ for all $x \leq y$ if and only if $P_{x,y'} = 1$ for all $x \leq y'$. The result then follows from Lemma~\ref{lem:KLgenNgen}.
\end{proof}

In~\cite[Section 2.1]{BurrullLibedinskyPlaza}, it is established that every element of $W$ is $\sim$-equivalent to an element in the disjoint union of four infinite families, denoted $X = \{ x_n \}_{n \in \N}$, $\Theta = \{ \theta(m,n) \}_{m,n \in \N}$, $\Theta_1 = \{ \theta(m,n)s_{m,n} \}_{m,n \in \N}$, and $\Theta_2 = \{ s_0 \theta(m,n) s_{m,n} \}_{m,n \in \N}$. We refer the reader to~\cite{BurrullLibedinskyPlaza} for the full definitions of these families, since we will only be interested in certain instances, which we define after the next statement.

\begin{thm}\label{thm:LP} 
Let $(W,S)$ be of type $\tilde{A}_2$.
\begin{enumerate}
\item $\KLbasis_{x_n} = \Ngen_{x_n}$ if and only if $n \leq 3$.
\item $\KLbasis_{\theta(m,n)} = \Ngen_{\theta(m,n)}$ if and only if at least one of $m$ and $n$ is equal to $0$.
\item If $y \in \Theta_1 \sqcup \Theta_2$, then $\KLbasis_{y} \neq \Ngen_{y}$.
\end{enumerate}
\end{thm}
\begin{proof} Parts (1) and (2) are immediate from the formulas for $\KLbasis_y$ given in parts (i) and (ii) of Theorem 1 of~\cite{LibedinskyPatimo}, respectively. The last part is then immediate from the more explicit formulas for $\KLbasis_y$ given for $y \in \Theta_1$ (respectively, $y \in \Theta_2$)  by  Proposition~3.1 (respectively, Proposition~3.3) of~\cite{BurrullLibedinskyPlaza}.\end{proof}

We now define certain of the elements of $W$ which appear in Theorem~\ref{thm:LP}.  Put $S = \{ s_0, s_1, s_2 \}$, where $\{s_1,s_2\}$ is the set of generators for the corresponding spherical Coxeter system of type~$A_2$. This notation is consistent with that of~\cite{BurrullLibedinskyPlaza}; whereas in~\cite{LibedinskyPatimo}, the generating set is instead denoted $\{ s_1, s_2, s_3\}$. In both~\cite{LibedinskyPatimo} and~\cite{BurrullLibedinskyPlaza}, the generator $s_i$ is sometimes  denoted by $i$, and $i$ is often taken modulo $3$.

\begin{figure}[ht]
 \resizebox{3in}{!}
 {
\begin{overpic}{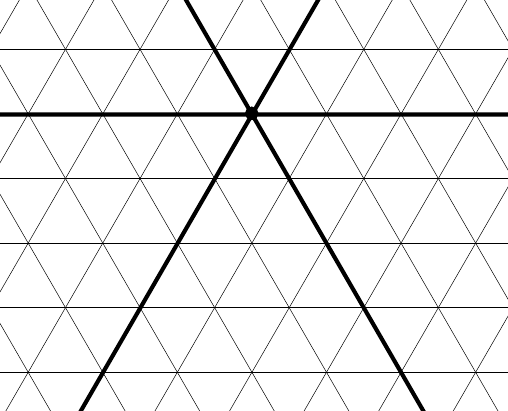}
\put(48,65){$x_0$}
\put(40,62){$x_1$}
\put(40,53){$x_2$}
\put(33,50){$x_3$}
\put(48,50){$y_0$}
\put(56,37){$y_1$}
\put(62,24){$y_2$}
\put(69,11){$y_3$}
\put(41,36){$\tilde{y}_1$}
\put(34,24){$\tilde{y}_2$}
\put(26,11){$\tilde{y}_3$}
\end{overpic}
}
\caption{\footnotesize{The Coxeter complex in type $\tilde{A}_2$, showing the elements $x_0 = 1$, $x_1 = s_1$, $x_2 = s_1 s_2$, $x_3 = s_1 s_2 s_0$, together with $y_m$ and $\tilde{y}_m$ for small $m$. The heavy lines are the hyperplanes bounding the Weyl chambers.}}
\label{fig:xy_alcoves}
\end{figure}

Returning to our notation, we have $x_0 = 1$, $x_1 = s_1$, $x_2 = s_1 s_2$, and $x_3 = s_1 s_2 s_0$. Then for $m \in \N$, the element $\theta(m,0)$ of $W$ is defined by the initial subword of length $3 + 2m$ of the reduced word
\[
(s_1 s_2 s_1) (s_0 s_2 s_1)^m.
\]
In order to simplify notation we will write $y_m = \theta(m,0)$. 
So $y_0 = s_1 s_2 s_1$ is the longest element in type $A_2$, and we have
\[
y_1 = (s_1 s_2 s_1) s_0 s_2, \quad y_2 = (s_1 s_2 s_1) s_0 s_2 s_1 s_0, \quad y_3 =  (s_1 s_2 s_1) s_0 s_2 s_1 s_0 s_2 s_1, 
\]
and so forth. The family $\{ \theta(0,n) \}_{n \in \N}$ also appears in the statement of Theorem~\ref{thm:LP}. However, we observe that the diagram automorphism which swaps $1$ and $2$ (in the notation of both~\cite{LibedinskyPatimo} and~\cite{BurrullLibedinskyPlaza}) swaps the elements $\theta(m,0)$ and $\theta(0,m)$, for all $m \in \N$; that is, $y_m \sim \theta(0,m)$ for every $m \in \N$. We thus denote by $\tilde{y}_m = \theta(0,m)$ for all $m \in \N$. See Figure~\ref{fig:xy_alcoves} for the locations of these elements of $W$ in the Coxeter complex.

\begin{corollary}\label{cor:reduction} 
Let $(W,S)$ be of type $\tilde{A}_2$. Suppose that for all
\[y \in \{ 1, s_1, s_1 s_2, s_1 s_2 s_0 \} \cup \{ y_m \}_{m \in \N}, \]
the graph $[1,y]_\cB$ can be cubulated. Then for every $y' \in W$ such that $P_{x,y'} = 1$ for all $x \leq y'$, the graph $[1,y']_\cB$ can be cubulated.
\end{corollary}

\begin{proof} 
Let $y' \in W$ be such that $P_{x,y'} = 1$ for all $x \leq y'$.  By Corollary~\ref{cor:diagram}, we have $\KLbasis_{y'} = \Ngen_{y'}$. Since $(W,S)$ is of type $\tilde{A}_2$, by Theorem \ref{thm:LP}, the observations about the four families of elements from \cite[Section 2.1]{BurrullLibedinskyPlaza}, and the fact that $y_m \sim \tilde{y}_m$ for all $m \in \N$, we thus have $y' \sim y$ for some $y \in \{ 1, s_1, s_1 s_2, s_1 s_2 s_0 \} \cup \{ y_m \}_{m \in \N}$. 
Now by assumption, $[1,y]_\cB$ can be cubulated. Since diagram automorphisms induce isomorphisms of Bruhat graphs, the graph $[1,y']_\cB$ can thus be cubulated as well.
\end{proof}

We thus have proved that in order to establish Theorem~\ref{thm:A2tilde}, it suffices to cubulate the infinite family of graphs $\{ [1,y_m]_\cB \}_{m \in \N}$, since the handful of special cases $[1,y]_\cB$ for $y \in \{ 1, s_1, s_1 s_2, s_1 s_2 s_0 \}$ are all treated by Lemma \ref{lem:boolean}.

\subsection{Construction of cubulations}\label{sec:infiniteA2tilde}

We now complete the proof of Theorem~\ref{thm:A2tilde}, by cubulating $[1,y_m]_\cB$ for every $m \in \N$. First, the element $y_0 = s_1 s_2 s_1$ is contained in a subsystem of type $A_2$, which is dihedral. Hence by \cref{prop:dihedral}, the graph $[1,y_0]_\cB$ is cubulated by $\cC(1,2)$. Our starting point for the cases $m \geq 1$ is the following observation.

\begin{lemma}\label{lem:sizeBruhatCubical} For any $m \geq 1$, the following sets are all of cardinality $3(m+1)(m+2)$:
\begin{enumerate}
\item the Bruhat interval $[1,y_m] \subset W$;
\item the vertex set of the graph $[1,y_m]_\cB$; and
\item the vertex set of the cubical lattice $\cC(2,m,m+1)$.
\end{enumerate}
\end{lemma}
\begin{proof} Part (1) is a special case of Lemma 1.4 of~\cite{LibedinskyPatimo}, and (2) follows by definition. From Definition~\ref{defn:cubical} it is clear that for $m \geq 1$, the cubical lattice $\cC(2,m,m+1)$ has \[ (2+1)(m+1)((m+1) + 1) = 3(m+1)(m+2) \] distinct vertices, which proves (3).
\end{proof}

The next statement, which completes the proof of Theorem~\ref{thm:A2tilde}, is motivated by the numerics in Lemma~\ref{lem:sizeBruhatCubical}. We will illustrate its proof by several figures, in which we write $s_{i_1 \cdots i_k}$ for the product $s_{i_1} \cdots s_{i_k}$, to save space.

\begin{prop}\label{prop:infinite} For all integers $m \geq 1$, the graph $[1,y_m]_\cB$ can be cubulated by $\cC(2,m,m+1)$.
\end{prop}

To prove this proposition, we first establish some useful terminology and notation. To simplify notation, we will sometimes put $\cC_m = \cC(2,m,m+1)$.  
For $k = 0,1,2$, we define the \emph{level $k$ vertices} of $\cC_m$, denoted $V_k(\cC_m)$,  to be all  vertices of $\cC_m$ with first coordinate $k$. That is,
\[
V_k(\cC_m) = \{ (k_1, k_2, k_3) \in V(\cC_m) \mid k_1 = k \}.
\]
Now for $k = 0,1,2$,  we define the  \emph{level $k$ edges} of $\cC_m$, denoted $E_k(\cC_m)$, to be all edges in $\cC_m$ which connect two elements of $V_k(\cC_m)$. The \emph{horizontal edges} of $\cC_m$ are then given by the (disjoint) union of the level $0$, $1$, and $2$ edges. Finally, the \emph{vertical edges} of $\cC_m$ are those from a level $k$ to a level $k + 1$ vertex, for $k = 0,1$. Thus in particular, every edge of $\cC_m$ is either horizontal or vertical (not both).

We can now formulate a precise (but rather technical) statement, which will be key to the proof of Proposition~\ref{prop:infinite}, as follows. 

\begin{lemma}\label{lem:label} 
For each $m \geq 1$, there is a bijection $\phi = \phi_m : V(\cC_m) \to [1,y_m]$ such that all of the following hold:
\begin{enumerate}
\item For every level $0$ edge $(v,v')$ in $\cC_m$, the graph $[1,y_m]_\cB$ has an edge $(\phi(v),\phi(v'))$.
\item For $k = 1,2$:
\begin{enumerate}
\item $\phi(k,k_2,k_3) = s_k \phi(k-1,k_2,k_3)$; and
\item $\ell(\phi(k,k_2,k_3)) = \ell(\phi(k-1,k_2,k_3)) + 1$.
\end{enumerate}
\end{enumerate}
\end{lemma}

\begin{figure}[h]
\begin{overpic}{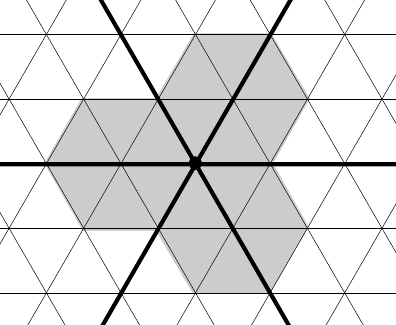}
\put(48,50){$1$}
\put(48,62){$s_0$}
\put(65,50){$s_{20}$}
\put(57,46){$s_2$}
\put(55,68){$s_{02}$}
\put(65,60){$s_{202}$}
\put(38,46){\textcolor{blue}{$s_1$}}
\put(26,50){\textcolor{blue}{$s_{10}$}}
\put(26,28){\textcolor{blue}{$s_{120}$}}
\put(35,34){\textcolor{blue}{$s_{12}$}}
\put(17,44){\textcolor{blue}{$s_{102}$}}
\put(16,34){\textcolor{blue}{$s_{1202}$}}
\put(56,34){\textcolor{red}{$s_{21}$}}
\put(64,28){\textcolor{red}{$s_{210}$}}
\put(44,19){\footnotesize{\textcolor{red}{$s_{2120}$}}}
\put(45,28){\textcolor{red}{$s_{212}$}}
\put(64,19){\footnotesize{\textcolor{red}{$s_{2102}$}}}
\put(57,12){\textcolor{red}{$y_1$}}
\end{overpic}
\caption{\footnotesize{The Bruhat interval $[1,y_1]$ for $y_1 = (s_1 s_2 s_1)s_0 s_2$, labeled to illustrate the proof of the case $m = 1$ in Lemma~\ref{lem:label}.}}
\label{fig:y1Bruhat}
\end{figure}

\begin{proof}
We carry out induction on $m \geq 1$. For $m = 1$, we have $y_1 = (s_1 s_2 s_1)s_0 s_2$, and that the interval $[1,y_1]$ contains $18$ elements, by Lemma~\ref{lem:sizeBruhatCubical}.  These 18 elements are shaded gray in Figure~\ref{fig:y1Bruhat}. We define $\phi$ on the level $0$ vertices of $\cC_1 = \cC(2,1,2)$ by:
 \[
 \phi(0,0,0) = 1, \quad \phi(0,0,1) = s_2 , \quad \phi(0,0,2) = s_2 s_0, \]\[ \phi(0,1,0) = s_0, \quad \phi(0,1,1) = s_0 s_2, \quad \phi(0,1,2) = s_2 s_0 s_2.
 \]
Identifying $\cC(1,2)$ with the subgraph $\cC(0,1,2)$ of $\cC_1$, this labeling of the level $0$ vertices is depicted on the left of Figure~\ref{fig:m=1}.
It is then straightforward to verify that the restriction of $\phi$ to $V_0(\cC_1)$ is a bijection onto a subset of $[1,y_1]$. This subset is labeled in black in Figures~\ref{fig:y1Bruhat} and~\ref{fig:m=1}. Since $\phi(V_0(\cC_1))$ is the set of elements of the parabolic subgroup of $W$ of type $A_2$ generated by $s_0$ and $s_2$, it is easy to see that part (1) holds.

\begin{figure}[h]
\begin{subfigure}[m]{.4\linewidth}\centering
		\includegraphics[page=6]{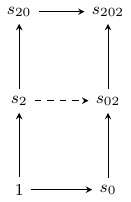}
\end{subfigure}
\begin{subfigure}[m]{.4\linewidth}\centering
		\includegraphics[page=7]{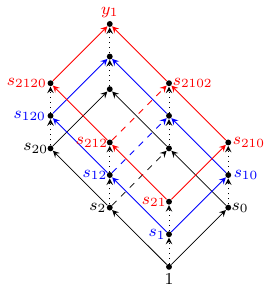}
\end{subfigure}
\caption{\footnotesize{The labeling $\phi$ of the vertices of $\cC(0,1,2)\cong \cC(1,2) $ on the left, and of all vertices of $\cC(2,1,2)$ on the right. On both the left and right, the solid edges correspond to right-multiplication by a simple generator (either $s_0$ or $s_2$) and the dashed edges correspond to right-multiplication by $s_{202}$. The dotted edges on the right correspond to left-multiplication by $s_k$ (to go from level $k$ to level $k+1$), for $k = 1,2$.}}
\label{fig:m=1}
\end{figure}

On the remaining vertices of $\cC_1$, we define $\phi(1,k_2,k_3) = s_1 \phi(0,k_2,k_3)$ and then $\phi(2,k_2,k_3) = s_2 \phi(1,k_2,k_3)$. From the description of the Bruhat intervals $[1,y_m] = [1,\theta(m,0)]$ developed in the introduction to~\cite{LibedinskyPatimo}, one can verify that $\phi$ is then a bijection $V(\cC_1) \to [1,y_1]$. In particular, for the unique vertex of $\cC_1 = \cC(2,1,2)$ of maximal rank we have 
\[
\phi(2,1,2) = s_2 \phi(1,1,2) = s_2 s_1 \phi(0,1,2) = s_2 s_1 s_2 s_0 s_2 = s_1 s_2 s_1 s_0 s_2 = y_1.
\]
In Figures~\ref{fig:y1Bruhat} and~\ref{fig:m=1}, the elements $\phi(1,k_2,k_3) = s_1 \phi(0,k_2,k_3)$ are labeled in blue, and the elements  $\phi(2,k_2,k_3) = s_2 \phi(1,k_2,k_3)$ are labeled in red.
Now (2)(a) holds by construction, and (2)(b) can be checked quickly. That is, for $m = 1$ we have constructed a bijection $\phi:V(\cC_1) \to [1,y_1]$ satisfying the statement.

\begin{figure}[ht!]
\begin{overpic}{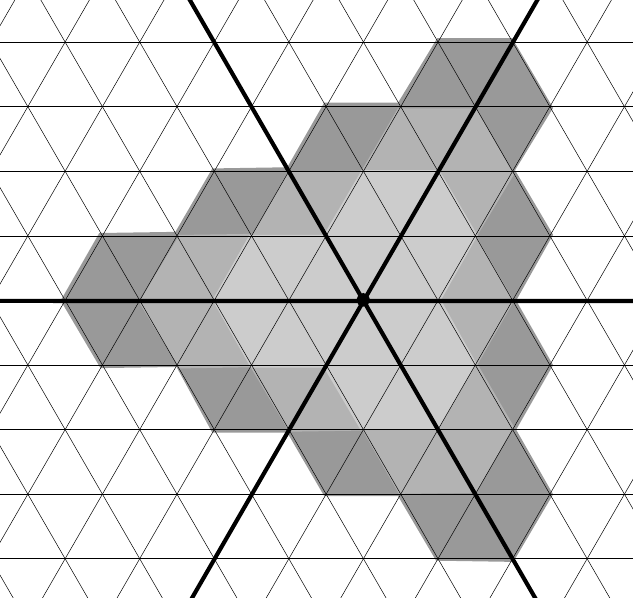}
\put(56.5,52){$1$}
\put(56,60){$s_0$}
\put(67,52){$s_{20}$}
\put(62,50){$s_2$}
\put(61,63){$s_{02}$}
\put(67,60){$s_{202}$}
\put(50,63){$s_{01}$}
\put(61,70){$s_{021}$}
\put(73,50){$s_{201}$}
\put(66,74){\footnotesize{$s_{0210}$}}
\put(72,63){\footnotesize{$s_{2021}$}}
\put(72,69){\footnotesize{$s_{02101}$}}
\put(78,53){\footnotesize{$s_{2012}$}}
\put(78,59){\footnotesize{$s_{20212}$}}
\put(77,75){\footnotesize{$s_{021012}$}}
\put(49,70){$s_{012}$}
\put(54.5,74){\footnotesize{$s_{0212}$}}
\put(66,79){\footnotesize{$s_{02102}$}}
\put(71,86){\footnotesize{$s_{021021}$}}
\put(77,79){\footnotesize{$s_{0210212}$}}
\put(50,50){\textcolor{blue}{$s_1$}}
\put(49.5,43){\textcolor{blue}{$s_{12}$}}
\put(61.5,43){\textcolor{red}{$s_{21}$}}
\put(55,39){\textcolor{red}{$s_{212}$}}
\put(62,30){\textcolor{red}{$y_1$}}
\put(68,20){\textcolor{red}{$y_2$}}
\put(74,10){\textcolor{red}{$y_3$}}
\end{overpic}
\caption{\footnotesize{The Bruhat intervals $[1,y_m]$ for $m = 1,2,3$, in successively darker shades of gray. The elements labeling the level $0$ vertices of $\cC_m$ are indicated in black, and certain elements at level 1 (respectively, level 2) are indicated in blue (respectively, red).}}\label{fig:y1y2y3Bruhat}
\end{figure}

\begin{figure}
\begin{subfigure}[m]{.45\linewidth}\centering
	\includegraphics[page=8]{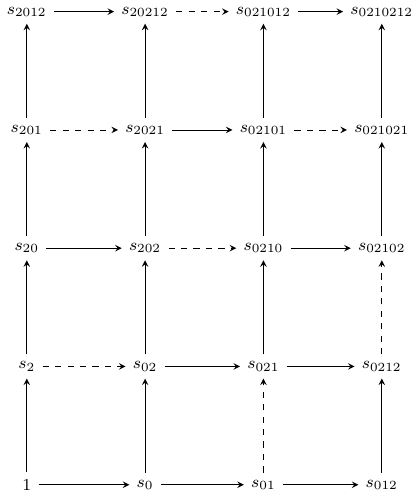}
\end{subfigure}
\begin{subfigure}[m]{.45\linewidth}\centering
	\includegraphics[page=9]{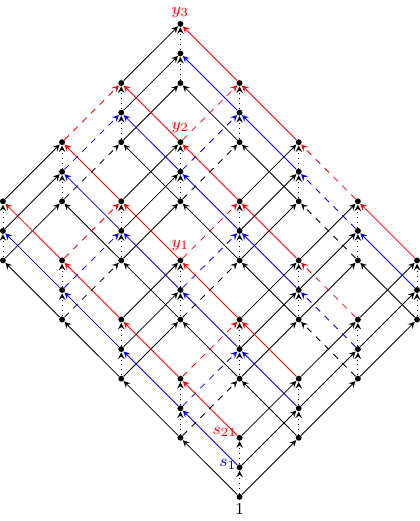}
\end{subfigure}
	\caption{\footnotesize{The labeling $\phi$ of the vertices of $\cC(0,3,4)\cong \cC(3,4) $ on the left, and of certain vertices of $\cC(2,3,4)$ on the right. On both the left and right, the solid edges correspond to right-multiplication by a simple generator, and the dashed edges to right-multiplication by the longest element in a subsystem of type $A_2$. On the right, the dotted vertical edges correspond to left-multiplication by $s_k$ (to go from level $k$ to level $k+1$), for $k = 1,2$. }}\label{fig:m=123}
\end{figure}

Assume by induction that we have a bijection $\phi_m:V(\cC_m) \to [1,y_m]$ satisfying the statement, for $m \geq 1$. The Bruhat intervals $[1,y_m]$ for $m = 1,2,3$ are depicted in Figure~\ref{fig:y1y2y3Bruhat}. We identify $\cC_m = \cC(2,m,m+1)$ with its natural image in $\cC_{m+1} = \cC(2,m+1,m+2)$, and define $\phi_{m+1}(v) = \phi_m(v)$ for all $v \in V(\cC_m)$. We now explain how to define $\phi_{m+1}$ on the remaining vertices of $\cC_{m+1}$. This construction will depend on the value of $m$ modulo $3$, and so we now work with simple generators $s_i$ where the subscript $i$ is taken modulo $3$. For all $0 \leq k_2 \leq m$, we define
\[
\phi_{m+1}(0,k_2,m+2) = \phi_m(0,k_2,m+1) s_m
\]
and for all $0 \leq k_3 \leq m$, we define
\[
\phi_{m+1}(0,m+1,k_3) = \phi_m(0,m,k_3) s_m.
\]
We then define $\phi_{m+1}$ on the remaining two level $0$ vertices of $\cC_{m+1}$ by
\[
\phi_{m+1}(0,m+1,m+1) = \phi_{m+1}(0,m+1,m) s_{m-1}\] and then \[ \phi_{m+1}(0,m+1,m+2) = \phi_{m+1}(0,m+1,m+1) s_m.
\]
Identifying $\cC(m,m+1)$ with the subgraph $\cC(0,m,m+1)$ of $\cC_m$, this labeling of the level~$0$ vertices is depicted in Figure~\ref{fig:m=123}, for $m = 3$.
Finally, as in the base case of the induction, we define $\phi_{m+1}(1,k_2,k_3) = s_1 \phi_{m+1}(0,k_2,k_3)$ and $\phi_{m+1}(2,k_2,k_3) = s_2 \phi_{m+1}(1,k_2,k_3)$, so that part (2)(a) of the statement holds by construction.

Using the description of the Bruhat interval $[1,y_m]$ from~\cite{LibedinskyPatimo}, it is straightforward to check that $\phi_{m+1}$ is a bijection from $V(\cC_{m+1})$ to $[1,y_{m+1}]$. It can also be easily verified from this description and our construction that if $(v,v')$ is a level $0$ edge which is in $\cC_{m+1}$ but not in $\cC_m$, then $\ell(\phi(v')) = \ell(\phi(v)) + 1$. Moreover, for all such edges $(v,v')$, the element $\phi(v')$ is obtained from $\phi(v)$ by either right-multiplication by a single generator, or right-multiplication by the longest element in a subsystem of type $A_2$. Therefore by induction, part (1) holds. Property (2)(b) can then be obtained without difficulty from the description in~\cite{LibedinskyPatimo}, by considering inversion sets for the relevant elements of $W$.
\end{proof}

We can now prove Proposition \ref{prop:infinite}, from which Theorem \ref{thm:A2tilde} immediately follows.

\begin{proof}[Proof of Proposition~\ref{prop:infinite}]
Fix $m \geq 1$ and let $\phi:V(\cC_m) \to [1,y_m]$ be as in Lemma~\ref{lem:label}. Since $\phi$ is a bijection from the vertex set of $\cC_m$ to the vertex set of $[1,y_m]_\cB$, we just need to see that for every edge $(v,v')$ in $\cC_m$, there is an edge $(\phi(v),\phi(v'))$ in $[1,y_m]_\cB$. Part (1) of Lemma~\ref{lem:label} gives this for the level $0$ edges of $\cC_m$. 

Now suppose $(v,v')$ is a level $1$ edge of $\cC_m$. Write $v = (1,k_2,k_3)$ and $v' = (1,k_2',k_3')$ and put $v_0 = (0,k_2,k_3)$ and $v_0' = (0,k_2',k_3')$. Then since $(v,v')$ is a horizontal edge of $\cC_m$, there must be a (horizontal) edge $(v_0, v_0')$. Hence by (1) of Lemma~\ref{lem:label}, there is an edge $(\phi(v_0),\phi(v_0'))$ in $[1,y_m]_\cB$.  This means exactly that there is a reflection $r$ such that $\phi(v_0) r = \phi(v_0')$ and $\ell(\phi(v_0)r) > \ell(\phi(v_0))$. Now by (2)(a) of Lemma~\ref{lem:label} with $k = 1$, we have $\phi(v) = s_1 \phi(v_0)$ and $\phi(v') = s_1 \phi(v_0')$. Thus \[\phi(v) r = s_1 \phi(v_0) r = s_1 \phi(v_0') = \phi(v').\] Using (2)(b) of Lemma~\ref{lem:label} with $k = 1$, we have $\ell(\phi(v)) = \ell(\phi(v_0)) + 1$ and $\ell(\phi(v')) = \ell(\phi(v_0')) + 1$. Hence
\[
\ell(\phi(v) r) =  \ell(\phi(v_0')) + 1 = \ell(\phi(v_0)r) +1 > \ell(\phi(v_0)) + 1 = \ell(\phi(v)).
\]
Thus there is an edge in $[1,y_m]_\cB$ from $\phi(v)$ to $\phi(v')$. The argument is similar for the level $2$ edges of $\cC_m$.

It remains to consider the vertical edges of $\cC_m$. Suppose $v_0 = (0,k_2,k_3)$ is a level $0$ vertex of $\cC_m$ and put $v_1 = (1,k_2,k_3)$, so that there is a vertical edge $(v_0,v_1)$. Then by (2)(a) of Lemma~\ref{lem:label}, we have $\phi(v_1) = s_1 \phi(v_0)$, and by (2)(b) of Lemma~\ref{lem:label}, we have $\ell(\phi(v_1)) = \ell(\phi(v_0)) + 1$. Hence in particular, $\ell(\phi(v_1)) > \ell(\phi(v_0))$. To see that $\phi(v_1)$ is obtained from $\phi(v_0)$ by right-multiplication by a reflection, let $r = \phi(v_0)^{-1} s_1 \phi(v_0)$. Then $r$ is a reflection and we have $\phi(v_1) = \phi(v_0) r$. Thus there is an edge $(\phi(v_0) , \phi(v_1))$ in $[1,y_m]_\cB$. A similar argument holds for the vertical edges from level~$1$ to level $2$. Thus $[1,y_m]_\cB$ is cubulated by $\cC_m$, as required.
\end{proof}

\MainTheoremD

\begin{proof}
Let $(W,S)$ be of type $\tilde{A}_2$. By Lemma~\ref{lem:boolean}, Proposition~\ref{prop:dihedral}, and Proposition~\ref{prop:infinite}, the graphs $[1,y]_\cB$ for all $y \in \{ 1, s_1, s_1 s_2, s_1 s_2 s_0 \} \cup \{ y_m \}_{m \in \N}$ can be cubulated. Therefore, Corollary~\ref{cor:reduction} says that whenever $P_{x,y'} = 1$ for all $x \leq y'$, the Bruhat graph $[1,y']_\cB$ can be cubulated.  In other words, the converse to Theorem~\ref{thm:Cubical=>Trivial} holds. 
\end{proof}


\appendix

\section{Cubulation of the Bruhat graph in types \texorpdfstring{$A$}{A} and \texorpdfstring{$B/C$}{B/C}}\label{sec:appendix}

In this appendix we give our explicit constructions of cubulations of the Bruhat graph $\cB = [1,w_0]_\cB$ in types $A$ and $B/C$. As discussed in the introduction, our constructions give an alternative approach to the Lehmer codes for the Coxeter systems of these types constructed in~\cite{Bolognini2025}: we make explicit use of normal form forests, and employ graph-theoretic arguments. In Section~\ref{sec:NF} we recall background material on normal forms, and then in Section~\ref{sec:NFF_cubulation} we give our constructions.

\subsection{Normal forms and forests}\label{sec:NF}

In this section, we follow \cite[Sec.~3.4]{BjoernerBrenti}, based upon the work of du Cloux~\cite{duCloux}, to review the construction of the normal form forest associated to a Coxeter system. Throughout, $(W,S)$ is an arbitrary Coxeter system, and we fix an indexing of the generating set $S = \{ s_1, \dots, s_n\}$.

A \emph{normal form} for $(W,S)$ is a specific choice of reduced expression for every element of $W,$ of which there are typically several systematic choices. 
The normal form we consider is the \emph{lexicographically first normal form}; that is, for each $x \in W$, we choose the reduced expression for $x$ which is first in the chosen lexicographic order on $S$. Denote this reduced expression by $\NF(x)$. Note that if $J \subseteq S$ and $x \in W_J$, where $\WJ$ denotes the standard parabolic subgroup generated by $J$, then $\NF(x)$ is the same as the lexicographically first normal form of $x$ regarded as an element of the subgroup $W_J$. Proposition \ref{prop:NF} below describes a factorization for this lexicographically first normal form. 

For any $J \subseteq S$, write $\JW$ for the set of minimal-length representatives of the right cosets $\WJ \backslash W$. An element $x \in W$ is in $\JW$ if and only if no reduced expression for $x$ begins with a letter from $J$. In the special case where $J = \{ s_1, \dots, s_j\} = \{ s_i \mid i \in [j] \}$ for some $j \in [n]$, we will write $\Wj$ for the subgroup $\WJ$ and $\jW$ for the set of minimal-length representatives of the cosets $\Wj \backslash W$. An element $x \in W$ is in $\jW$ if and only if no reduced expression for $x$ begins with a letter $s_i$ where $i \in [j]$. It will be convenient to define $W_\emptyset = W_{[0]}$ to be the trivial subgroup of $W$, in which case $\prescript{[0]}{}W = W$. Since $W_{[j-1]} \leq \Wj$ for all $j \in [n]$, we denote by $\prescript{[j-1]}{}{(\Wj)}$ the set of minimal-length representatives of the right cosets $W_{[j-1]} \backslash \Wj$. An element $x \in \Wj$ is in $\prescript{[j-1]}{}{(\Wj)}$ if and only if every reduced expression for $x$ begins with $s_j$.

\begin{prop}[Proposition 3.4.2 of~\cite{BjoernerBrenti}]\label{prop:NF} 
Let $(W,S)$ be any Coxeter system. 
Any $x \in W$ can be written uniquely as $x = x_1 \cdots x_n$ where $x_j \in \prescript{[j-1]}{}{(W_{[j]})}$ for all $j \in [n]$. Moreover, $\NF(x) = \NF(x_1) \cdots \NF(x_n).$
\end{prop}

The \emph{normal form forest} of $(W,S)$ consists of edge-labeled rooted trees $\tau_1,\dots,\tau_n$, with vertices and edges defined as follows. For a fixed $j \in [n]$, the vertices of $\tau_j$ correspond to the elements of the set $\prescript{[j-1]}{}{(W_{[j]})}$.  The edges of $\tau_j$ are labeled by elements of $S$ such that the (unique) path from the root of $\tau_j$ to the vertex $x_j \in \prescript{[j-1]}{}{(W_{[j]})}$ is labeled by the reduced expression $\NF(x_j)$. Then by Proposition~\ref{prop:NF}, the set of all normal forms $\{ \NF(x) \mid x \in W \}$ is obtained by concatenating all (possibly empty) rooted paths in $\tau_1, \dots, \tau_n$, in this order. 

\begin{figure}[h]
 \resizebox{1.5in}{!}
 {
\begin{overpic}{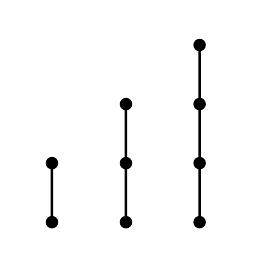}
\put(16,5){$\tau_1$}
\put(43.5,5){$\tau_2$}
\put(71.5,5){$\tau_3$}
\put(13,26){$1$}
\put(40,26){$2$}
\put(67,26){$3$}
\put(40,48){$1$}
\put(67,48){$2$}
\put(67,70){$1$}
\end{overpic}
}
\caption{\footnotesize{The normal form forest for type $A_3$ obtained by labeling nodes of the Dynkin diagram from left to right.}}
\label{fig:A3NFF}
\end{figure}

\begin{example}\label{ex:A3NFF}
Let $(W,S)$ have type $A_3$, in which case $\prescript{[0]}{}W = W \cong S_4$.  Label the nodes of the Dynkin diagram from left to right. We then have that $\prescript{[0]}{}{(W_{[1]})}$ is the set of minimal-length representatives of the right cosets $W_{[0]}\backslash W_{[1]} = W_{[1]}$, where $W_{[1]} = \langle s_1 \rangle$. The tree $\tau_1$ is thus labeled by 1, and the vertices of $\tau_1$ correspond to the elements $\{1, s_1\}$.

Now consider the set of minimal-length representatives of the right cosets $W_{[1]} \backslash W_{[2]}$, where $W_{[2]} = \langle s_1, s_2 \rangle$. As every reduced expression in $\prescript{[1]}{}{(W_{[2]})}$ begins with $s_2$, the edge label immediately above the root vertex of $\tau_2$ is 2.  The number of minimal length elements in $W_{[1]} \backslash W_{[2]}$ equals $|W_{[2]}| / |W_{[1]}| = 6/2 = 3$. The root and the adjacent vertex correspond to $\{1,s_2\}$, so there is one remaining vertex, which must be labeled by 1 in order that the elements of $W_{[1]} \backslash W_{[2]}$ all begin with $s_2$.  The tree $\tau_2$ is thus labeled by 2, then 1, as shown in Figure \ref{fig:A3NFF}, with the vertices corresponding to the elements $\{1, s_2, s_2s_1\}$.

This pattern continues, such that the label above the root vertex of $\tau_3$ is 3. The additional edge-labels are given by the remaining indices listed in decreasing order, so that the non-identity elements of $\prescript{[2]}{}{(W_{[3]})} = \{1, s_3, s_3s_2, s_3s_2s_1\}$ necessarily begin with $s_3$; see Figure \ref{fig:A3NFF}.

The normal form for $w_0$, for example, given by applying Proposition \ref{prop:NF} to this normal form forest, is then $w_0 = s_1 \cdot s_2s_1 \cdot s_3s_2s_1$, obtained by concatenating the expressions corresponding to each of the paths $\tau_1, \tau_2, \tau_3$, where we read upward from the root of each tree.
\end{example}

\subsection{Normal form forests and cubulation}\label{sec:NFF_cubulation}

In this section, we prove that the Bruhat graph $\cB = [1,w_0]_\cB$ is spanned by a cubical lattice whenever $W$ is of type $A$ or $B/C$. In Section~\ref{sec:paths} we construct a cubulation of $\cB$ whenever each tree in a normal form forest (NFF) is a path. Then in Section~\ref{sec:classical} we prove that such a NFF exists in types $A$ and $B/C$.


\subsubsection{Normal form forests of paths}\label{sec:paths}

In this section, we prove that if each tree in a normal form forest from Section \ref{sec:NF} is a path, then the Bruhat graph $\cB = [1,w_0]_\cB$ can be cubulated.

\begin{prop}\label{prop:NFcubical} 
Let $(W,S)$ be a finite Coxeter system with longest element $w_0 \in W$. Suppose $(W,S)$ has a normal form forest $\tau_1,\dots,\tau_n$ in which every rooted tree $\tau_i$ is the path consisting of $\ell_i \geq 1$ edges. Then $\cB = [1,w_0]_\cB$ is cubulated by $\cC(\ell_1,\dots,\ell_n)$.
\end{prop}

\begin{figure}[ht!]
	\begin{subfigure}[m]{.4\linewidth}\centering
		\includegraphics[page=4]{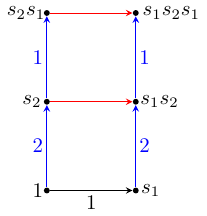}
	\end{subfigure}
	~
	\begin{subfigure}[m]{.4\linewidth}\centering
		\includegraphics[page=5]{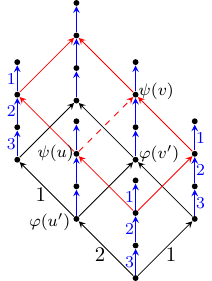}
	\end{subfigure}
	\caption{\footnotesize{On the left (respectively, right), we depict the inductive construction in the proof of Proposition~\ref{prop:NFcubical} in type $A_n$ for $n = 2$ (respectively, $n = 3$).}}\label{fig:w0}
\end{figure}

We will illustrate the proof of \Cref{prop:NFcubical} by following examples in types $A_2$ and $A_3$, as depicted in Figure~\ref{fig:w0}. Recall from Example~\ref{ex:A3NFF} that $(W,S)$ of type $A_3$ has a normal form forest of trees $\tau_1$, $\tau_2$, $\tau_3$, such that for $i = 1,2,3$ the tree $\tau_i$ is the path consisting of $i$ edges. The trees $\tau_1$ and $\tau_2$ give a normal form forest of paths in type $A_2$. We will see that in type~$A_2$ (respectively, $A_3$), the Bruhat graph $\cB$ is cubulated by $\cC(1,2)$ (respectively, $\cC(1,2,3)$).

\begin{proof}[Proof of \Cref{prop:NFcubical}] We proceed by induction on $n$. If $n = 1$, then $(W,S)$ is of type $A_1$ and the graph $[1,w_0]_\cB$ is isomorphic to $\cC(1)$. 

For the inductive step, write $\cB_{n-1}$ for the Bruhat graph of the subsystem of $(W,S)$ generated by $\{ s_1, \dots, s_{n-1}\}$, and assume that there is an isomorphism $\varphi$ from $\cC_{n-1}= \cC(\ell_1, \dots, \ell_{n-1})$ to a subgraph of $\cB$ which spans $\cB_{n-1}$. 
 For example, in type $A_n$ with $n = 2$ (respectively, $n=3$), the graph $\varphi(\cC_{n-1})\cong \cC_{n-1}$ is shown in black on the left (respectively, right) of Figure~\ref{fig:w0}.

Now, using Proposition~\ref{prop:NF}, we have that every $x \in W$ factors uniquely as $x = x' x_n$, where $x' \in V(\varphi(\cC_{n-1}))$ and $x_n$ is the label of a vertex in the path $\tau_n$. For every fixed $x' \in V(\varphi(\cC_{n-1}))$, we define $\tau_{x'}$ to be the subgraph of $\cB$ induced by the vertex set $\{ x' x_n \mid x_n \mbox{ a vertex label in }\tau_n  \}$. Then each such $\tau_{x'}$ is naturally isomorphic to the path $\tau_n$, with this isomorphism preserving all edge-labels, and the paths $\{ \tau_{x'} \mid x' \in V(\varphi(\cC_{n-1})) \}$ are pairwise disjoint subgraphs of~$\cB$.  Moreover, the Bruhat graph $\cB$ is spanned by the union of $\varphi(\cC_{n-1})$ and the collection of paths $\{ \tau_{x'} \mid x' \in V(\varphi(\cC_{n-1})) \}$. In type $A_2$ (respectively, $A_3$), the paths $\{\tau_{x'}\}$ are depicted vertically in blue on the left (respectively, right) of Figure~\ref{fig:w0}.

Observe that the cubical lattice $\cC_{n-1}$ naturally embeds as the subgraph $\cC(\ell_1,\dots,\ell_{n-1},0)$ of $\cC_n = \cC(\ell_1, \dots, \ell_{n})$. For each vertex $v' = (m_1, \dots, m_{n-1})$ of $\cC_{n-1}$, we now define $\tau_{v'}$ to be the subgraph of $\cC_n$ induced by the vertex set $\{ (m_1,\dots,m_{n-1},m_n) \mid 0 \leq m_n \leq \ell_n \}$. Then each such $\tau_{v'}$ is naturally isomorphic to the path $\tau_n$, and the paths $\{ \tau_{v'} \mid v' \in V(\cC_{n-1})\}$ are pairwise disjoint in $\cC_n$. Write $\cD_n$ for the subgraph of $\cC_n$ which is the union of its naturally embedded copy of $\cC_{n-1}$, together with all of the paths $\{ \tau_{v'} \mid v' \in V(\cC_{n-1})\}$. Then $\cD_n$ is a spanning subgraph of $\cC_n$, and we can construct an isomorphism $\psi$ from $\cD_n$ onto a spanning subgraph of $\cB$ by sending the copy of $\cC_{n-1}$ in $\cD_n$ onto $\varphi(\cC_{n-1})$ and, for any $v' \in V(\cC_{n-1})$, mapping the subgraph $\tau_{v'}$ of $\cC_n$ onto the subgraph $\tau_{\varphi(v')}$ of $\cB$. That is, in type $A_n$ for $n = 2$ (respectively, $n = 3$), the graph $\psi(\cD_n) \cong \cD_n$ is the union of the black and blue edges on the left (respectively, right) of Figure~\ref{fig:w0}.

It now suffices to show that this map $\psi:\cD_n \to \cB$ can be extended to an embedding of the entire cubical lattice $\cC_n$ into $\cB$. This amounts to showing that for any edge $(u,v)$ of $\cC_n$ which is not already in~$\cD_n$, there is an edge of $\cB$ from $\psi(u)$ to $\psi(v)$. On both sides of Figure \ref{fig:w0}, we thus need to show that we can connect all pairs of vertices in $\cB$ which are at the same height and are such that the vertices directly underneath them in $\varphi(\cC_{n-1})$ are connected by a black edge. We depict the two such ``missing'' edges in red on the left of Figure~\ref{fig:w0}, and (to avoid cluttering the image) only some of the ``missing'' edges in red on the right of Figure~\ref{fig:w0}.

Let $u = (p_1, \dots, p_{n-1}, p_n)$ and $v = (q_1, \dots, q_{n-1},q_n)$ be vertices of $\cC_n$ such that $(u,v) \in E(\cC_n) \setminus E(\cD_n)$. Then $u$ and $v$ must have the same final component $p_n = q_n \geq 1$, and there must be an edge from $u' = (p_1, \dots, p_{n-1})$ to $v' = (q_1, \dots, q_{n-1})$ in $\cC_{n-1}$. 
Now since  $(u',v')$ is an edge of $\cC_{n-1}$, by induction $(\varphi(u'),\varphi(v'))$ is an edge of $\cB_{n-1} \subset \cB$. This means exactly that $\ell(\varphi(v')) > \ell(\varphi(u'))$ and there is a reflection $t' \in \cT$ so that $\varphi(v') = \varphi(u') t'$. 
For example, in type $A_3$ consider the vertices $u = (0,1,2)$ and $v = (1,1,2)$ of $\cC_3 = \cC(1,2,3)$. The dashed red edge from $\psi(u)=s_2(s_3s_2)$ to $\psi(v)=s_1s_2(s_3s_2)$, on the right of Figure~\ref{fig:w0}, corresponds to the black edge from $\varphi(u') = s_2$ to $\varphi(v') = s_1s_2$, and we have $t' = s_2 s_1 s_2$.

To see that there is an edge of $\cB$ from $\psi(u)$ to $\psi(v)$, let $x_n$ be the label of the vertex at the end of the initial subpath of $\tau_n$ which consists of $p_n = q_n \geq 1$ edges. (For example, on the right of Figure~\ref{fig:w0} we have $x_3 = s_3 s_2$.) 
Then by definition of $\psi$, we have $\psi(u) = \psi(p_1, \dots, p_{n-1},p_n) = \varphi(u') x_n$, and similarly $\psi(v) = \varphi(v') x_n$. Since $\varphi(u') \in V(\cB_{n-1})$ and $x_n$ is the label of a vertex of $\tau_n$, by Proposition~\ref{prop:NF} we obtain \[ \ell(\psi(u)) = \ell(\varphi(u')x_n) = \ell(\varphi(u')) + \ell(x_n) = \ell(\varphi(u')) + p_n,\] and similarly $\ell(\psi(v)) = \ell(\varphi(v')) + q_n$. From the previous paragraph, we have $\ell(\varphi(v')) > \ell(\varphi(u'))$ and $p_n = q_n$. Thus we obtain  $\ell(\psi(v)) > \ell(\psi(u))$, as desired. 
Now let $t \in \cT$ be the reflection given by $t = x_n^{-1} t' x_n$. (For example, on the right of Figure~\ref{fig:w0} we have $t = s_2 s_3 (s_2 s_1 s_2) s_3 s_2$.) Then  \[ \psi(v) = \varphi(v') x_n = (\varphi(u') t') x_n = \varphi(u_0) x_n (x_n^{-1} t' x_n) = \psi(u)t.\]  
Therefore $(\psi(u),\psi(v))$ is an edge of $\cB$, which completes the proof. Geometrically, the reflection~$t$ which labels the edge $(\psi(u),\psi(v))$ is obtained by going down the path $\tau_{\varphi(u')}$, then across the edge $(\varphi(u'),\varphi(v'))$, then up the path $\tau_{\varphi(v')}$. \end{proof}


\subsubsection{Normal form forests of paths in types $A$ and $B/C$}\label{sec:classical}

Motivated by Proposition~\ref{prop:NFcubical}, we now, for types $A$ and $B/C$, exhibit normal form forests in which every rooted tree is a path. Hence the Bruhat graph $\cB = [1,w_0]_\cB$ can be cubulated in these types.

\emph{Type $A_n$.} Suppose that $(W,S)$ is of finite type $A_n$ for $n \geq 1$.  We shall construct a normal form forest $\tau_1, \dots, \tau_n$ in which every rooted tree $\tau_i$ is a path consisting of $i$ edges, generalizing Example \ref{ex:A3NFF} illustrated by Figure \ref{fig:A3NFF}.  

Label the nodes of the Dynkin diagram $1, \dots, n$ from left to right. Fix any $j \in [n]$, and consider the set $\prescript{[j-1]}{}{(W_{[j]})}$ of minimal-length representatives of the right cosets $W_{[j-1]} \backslash W_{[j]}$, where $W_{[j]} = \langle s_1, s_2, \dots, s_j \rangle$. In type $A_j$, the number of elements in $W_{[j-1]} \backslash W_{[j]}$ equals $|W_{[j]}| / |W_{[j-1]}| = \frac{j!}{(j-1)!} = j$. As every reduced expression in $\prescript{[j-1]}{}{(W_{[j]})}$ begins with $s_j$, the edge label immediately above the root vertex of $\tau_j$ is $j$. 

In order that all other reduced expressions read from $\tau_j$ also necessarily begin uniquely with $s_j$, we must label the remaining $j-1$ edges in turn by $j-1, j-2, \dots, 1$, due to the commuting relations encoded by the Dynkin diagram. Moreover, the expression $s_js_{j-1} \cdots s_2s_1$ is clearly reduced, as a Coxeter element in the parabolic subgroup $W_{[j]}$. Each of the $j$ initial subexpressions $s_j \cdots s_k$ for $k \in [j]$ is thus reduced and represents a distinct element of $\prescript{[j-1]}{}{(W_{[j]})}$. Therefore by Proposition \ref{prop:NF}, the rooted tree $\tau_j$ is the path consisting of $j$ edges which are labeled successively by $j, j-1, \dots, 1$.

By Proposition \ref{prop:NFcubical}, this construction of a normal form forest consisting of paths of length $1, 2, \dots, n$ proves that  $[1,w_0]_\cB$ is spanned by the cubical lattice $\cC(1, 2, \dots, n)$ in type $A_n$.

\emph{Types $B_n$ and $C_n$.} Suppose that $(W,S)$ is of finite type $B_n$ for $n \geq 2$. We shall construct a normal form forest $\tau_1, \dots, \tau_n$ in which every rooted tree $\tau_i$ is a path consisting of $2i-1$ edges.

Label the nodes of the Dynkin diagram $1, \dots, n$ from right to left, so that the last $n-1$ nodes form a type $A_{n-1}$ subsystem, and the special node is indexed by $1$; note that this is the reverse of the ordering from \cite{Bourbaki4-6}. Using this labeling, $W_{[1]} = \langle s_1 \rangle$ has type $A_1$, and $W_{[j]} = \langle s_1, s_2, \dots, s_j \rangle$ has type $B_j$ for all $j \geq 2$.  

We proceed by induction on $n \geq 2$.  First note that $\prescript{[0]}{}{(W_{[1]})} = \langle s_1\rangle$ has type $A_1$, and so $\tau_1$ is the path consisting of one edge labeled by 1, as seen in Figure \ref{fig:A3NFF}.  The group $W_{[2]} = \langle s_1, s_2 \rangle$ has type $B_2$, and so there are exactly 4 distinct minimal-length coset representatives in $\prescript{[1]}{}{(W_{[2]})} = \{ 1, s_2, s_2s_1, s_2s_1s_2\}$.  Since $s_2s_1s_2$ is reduced and all other elements of $\prescript{[1]}{}{(W_{[2]})}$ are initial subexpressions of $s_2s_1s_2$, then $\tau_2$ is a tree with 3 edges, labeled successively by $2, 1, 2$, establishing the base case for $n=2$.

Now for any $j \geq 3$, suppose that the $j-1$ rooted trees for the normal form forest of type $B_{j-1}$ are paths consisting of $2(j-1)-1$ edges, labeled successively by $1, 3, \dots, 2(j-1)-1$.  The group $W_{[j]}$ has type $B_j$, so that $\prescript{[j-1]}{}{(W_{[j]})}$ has $\frac{2^j j ! }{ 2^{j-1} (j-1)!} = 2j$ minimal-length coset representatives. By Theorem 1.1 of \cite{MilRWs} with the ordering of the nodes reversed, the expression $t = s_js_{j-1} \cdots s_2s_1s_2 \cdots s_{j-1}s_j$ is reduced, and therefore so are all $2j-1$ of its initial subexpressions.  Note by construction that both $t$ and all initial subexpressions must begin with $s_j$.  Together with the identity, we have thus identified the $2j$ elements of $\prescript{[j-1]}{}{(W_{[j]})}$ as those words formed by reading the labels of the tree $\tau_j$ with $2j-1$ edges, labeled successively by $j, j-1, \dots, 2,1,2, \dots, j-1, j$.

By induction and Proposition \ref{prop:NFcubical}, this construction of a normal form forest consisting of paths of lengths $1, 3, 5, \dots, 2n-1$ proves that  $[1,w_0]_\cB$ is spanned by the cubical lattice $\cC(1, 3, 5, \dots, 2n-1)$. Since the group $W$ is identical, the result follows for type $C_n$ as well.


\bibliographystyle{alphaurl}
\bibliography{bibliographyKL.bib}

\end{document}